\documentclass [11pt,reqno]{amsart}
\usepackage{amsmath,amssymb,verbatim,color}
\usepackage[all]{xy}

\usepackage[tmargin=1in, bmargin=1in, rmargin=1in, lmargin=1in]{geometry}
\usepackage{mathrsfs}

\usepackage[backref,pagebackref,hyperindex]{hyperref}
\usepackage{graphicx}

\def\XXint#1#2#3{{\setbox0=\hbox{$#1{#2#3}{\int}$ }
\vcenter{\hbox{$#2#3$ }}\kern-.6\wd0}}

\newcommand{\C}{\mathbb{C}}

 \newcommand{\R}{\mathbb{R}}
 
\newcommand{\T}{\mathbb{T}}

\newcommand{\ft}{\mathfrak{t}}

\newcommand{\cA}{\mathcal{A}}

\newcommand{\cE}{\mathcal{E}}

\newcommand{\cH}{\mathcal{H}}

\newcommand{\cM}{\mathcal{M}}

\newcommand{\tmom}{\widetilde{\mathcal{PM}}_{\Omega}}

\newcommand{\f}{\varphi}

\newcommand{\om}{\omega}

\newcommand{\p}{\psi}

\newcommand{\cET}{\cE^{1,T}}

\newcommand{\enR}{\mathrm R}

\DeclareMathOperator{\ent}{H}

\DeclareMathOperator{\Aut}{Aut}

\DeclareMathOperator{\Ent}{Ent}

\DeclareMathOperator{\MA}{MA}

\DeclareMathOperator{\PSH}{PSH}

\DeclareMathOperator{\tr}{tr}

\DeclareMathOperator{\Ric}{Ric}

\DeclareMathOperator{\dd}{{d}}

\DeclareMathOperator{\enE}{\mathrm E}

\newcommand{\ddc}{dd^c}
\newcommand{\dc}{d^c}
\newcommand{\de}{d}

\newcommand{\dbar}{\overline{\partial}}

\numberwithin{equation}{section}       

\newtheorem{prop} {Proposition} [section]
\newtheorem{thm}[prop] {Theorem} 
\newtheorem{defi}[prop] {Definition}
\newtheorem{lemma}[prop] {Lemma}
\newtheorem{cor}[prop]{Corollary}
\newtheorem{prop-def}[prop]{Proposition-Definition}

\newtheorem{mainthm}{Theorem}

\newtheorem{remark}[prop]{Remark}

\newtheorem{conj}[prop]{Conjecture}
\theoremstyle{remark}

\title{On Weighted Twisted K-Energy and Its Applications}
\author{Xia Xiao}
\date{\today}

\begin{document}

\begin{abstract}
We establish the convexity of the weighted twisted Mabuchi K-energy functional along geodesics in the finite-energy space $\mathcal{E}^{1,T}(X,\omega)$, covering the case of divisors with mixed cusp and conic singularities. We then prove that coercivity (relative to the complex torus) of this functional is an open condition under cone angle perturbations. This is obtained from a general result of independent interest, which shows the stability of coercivity under perturbations by certain twist currents. In particular, this yields openness of coercivity for conic K-energy functionals and proves that coercivity at the cusp limit implies coercivity of the corresponding conic K-energy functionals for all sufficiently small cone angles.
\end{abstract}
\maketitle

\section{Introduction}
Finding canonical K\"ahler metrics is a central theme in complex geometry. In higher dimensions, constant scalar curvature K\"ahler (cscK) metrics provide an analogue of Poincar\'e's uniformization theorem; see the surveys \cite{Boucksom_2018, Donaldson_2018_2} for background and further references. K\"ahler--Einstein metrics form an especially prominent subclass and have driven many breakthroughs, from the solution of the Calabi conjecture to the proof of the Yau--Tian--Donaldson (YTD) conjecture for Fano manifolds. For general polarizations $(X,L)$, recent variational advancements \cite{Darvas_Rubinstein_2017, BDL, CC1, CC2} relate cscK existence to coercivity of the Mabuchi K-energy. Li's work \cite{Li22} further connects the variational picture to stability through geodesic rays and model filtrations, and these ideas have recently led to a variational, Yau--Tian--Donaldson-type correspondence in a number of settings; see \cite{DarvasZhang2025, BoucksomJonsson2025}.

Fundamental to these existence and uniqueness results is the convexity of the K-energy. Initially established by Berman--Berndtsson \cite{BB} for $C^{1,1}$-geodesics, this property was extended by Berman--Darvas--Lu \cite{BDL} to singular potentials in finite energy spaces using pluripotential theory. While recent literature has introduced \emph{weighted} cscK metrics \cite{Lah, AJL}, a comprehensive variational theory---specifically regarding K-energy convexity for metrics that are simultaneously weighted and twisted in the presence of singularities---remains underdeveloped.

We establish convexity of weighted twisted K-energy functionals along geodesics in the space of finite energy potentials $(\cET(X,\omega), d_1)$ invariant under a fixed compact torus $T$. Our framework simultaneously incorporates weight functions and twisting by positive (1,1)-currents satisfying a natural decomposition condition. This result provides a unified foundation for variational approaches to weighted cscK metrics with general twist, extending the scope of the analytic Yau--Tian--Donaldson correspondence program.

\textbf{Setup and notation.}
Let $X$ be a smooth compact complex $n$-dimensional manifold and let $(X, \omega)$ be a compact K\"ahler manifold with a fixed reference K\"ahler metric $\omega$. Let $T \subset \Aut_{red}(X)$ denote a fixed compact torus in the \emph{reduced} group $\Aut_{red}(X)$ of automorphisms of $X$, i.e., the connected subgroup of automorphisms of $X$ generated by the Lie algebra of real holomorphic vector fields with zeros (see e.g.~\cite{gauduchon-book}). Consider a $T$-invariant positive $(1,1)$-current $\chi$ satisfying the following decomposition condition:
\begin{equation}\label{eq: ChiProp}
\chi = \beta + \frac{1}{2}  \ddc  f+\ddc \hat{f}, \textup{ where } e^{-f} \in L^{1}(X,\om^n)\footnote{The factor $1/2$ appearing here is consistent with our convention for the weighted Ricci curvature $\Ric_{\mathrm{v}}$ (see Definition \ref{def:equivariant-ricci}). In the conventions of \cite{BDL}, our twisted current $\chi$ is equivalent to $2\chi$ in their notations.},
\end{equation}
where $\beta$ is a smooth $(1,1)$-form, $f$ is an upper semicontinuous function, and $\hat{f} \in \cET(X, \omega)$. We assume $T \subset \Aut_{red}(X, \chi)$, where $\Aut_{red}(X,\chi)$ denotes the subgroup of $\Aut_{red}(X)$ preserving $\chi$ (cf.~\cite{Fuj}).

It is well known that $T$ acts in a Hamiltonian way with respect to any $T$-invariant K\"ahler metric in the K\"ahler class $[\omega]$, and the corresponding moment map $m_{\omega_\varphi}$ (for $\omega_\varphi = \omega + \ddc \varphi$) sends $X$ onto a compact convex polytope $P_{X,T} \subset \ft^\vee$ in the dual vector space $\ft^\vee$ of the Lie algebra $\ft$ of $T$ (cf.~\cite{Atiyah,GS}). Furthermore, up to translations, $P_{X,T}$ is independent of the choice of K\"ahler potential $\varphi$. Since this polytope depends on the choice of torus $T$, we denote it by $P_{X,T}$ and, once $T$ is fixed, simply by $P_X$. We shall fix this polytope $P_X$, giving rise to a normalization of the corresponding moment maps $\{m_{\omega_\varphi}\}$.

Using Lahdili's weighted scalar curvature formalism as the smooth weighted input, we consider its twisted finite-energy version. Let $\mathrm{v},\mathrm{w}\in C^\infty(\ft^\vee)$ be smooth functions with $\mathrm{v}>0$. We consider the \emph{weighted twisted scalar curvature equation}
\begin{equation}\label{eq:wtscK}
S_{\mathrm{v}}^{\chi}(\omega_\varphi) = \mathrm{w}(m_{\omega_\varphi}),
\end{equation}
where $S_{\mathrm{v}}^{\chi}$ denotes the weighted twisted scalar curvature (see Section~\ref{sec:wen} for the precise definition). The associated Euler-Lagrange functional is the \emph{weighted twisted Mabuchi functional} $\cM_{\mathrm{v},\mathrm{w}}^{\chi}$, which admits a natural decomposition into three components: the weighted entropy term $\ent_{\mathrm{v}}$, the twisted Ricci energy term $\enR_{\mathrm{v}}^{\chi}$, and the weighted energy term $\enE_{\mathrm{v}\,\mathrm{w}}$. Here $\cET(X, \omega)$ denotes the space of $T$-invariant finite energy potentials, which is the metric completion of the space of smooth $T$-invariant K\"ahler potentials with respect to the $d_1$-metric.

To clarify the scope of our framework, we note that it unifies several distinct geometric problems within a single variational setting:
\begin{itemize}
    \item \emph{Constant Scalar Curvature K\"ahler (cscK) metrics}: Setting $\mathrm{v}=1, \chi=0$ and $\mathrm{w}$ to be a constant recovers the classical cscK problem.
    
    \item \emph{Extremal K\"ahler metrics}: The choice $\mathrm{v}=1, \chi=0$ and $\mathrm{w}=\ell$, where $\ell$ is an affine-linear function on $\ft^\vee$, corresponds to the problem of extremal K\"ahler metrics. Our framework extends this to the singular setting, covering both conic and Poincar\'e-type extremal metrics (characterized by the extremal function $\ell^{\mathrm{ext}}$ discussed in Section~\ref{subsec:futaki}).
    
    \item \emph{Generalized K\"ahler--Ricci solitons ($\mu$-cscK)}: Let $\ell$ be an affine-linear function on $\ft^\vee$ and $a \in \R$. The choice $\mathrm{v}= e^{\ell}$ and $\mathrm{w}= 2(\ell + a) e^{\ell}$ corresponds to the $\mu$-cscK metrics introduced in~\cite{Ino}. When $\chi=0$, this generalizes the theory of K\"ahler--Ricci solitons on Fano manifolds~\cite{TZ}. In our twisted setting where $\chi$ represents a divisor with varying cone angles, this setup captures $\mu$-cscK metrics with cone singularities.
    
    \item \emph{Scalar-flat metrics on line bundles}: Let $n = \dim_\C X$ and let $\ell$ be a positive affine-linear function on $P_X$. Setting $\mathrm{v}=\ell^{-n-1}$ and $\mathrm{w}=a \ell^{-n-2}$ relates to the construction of scalar-flat K\"ahler metrics on the total space of a polarization $L$ (see~\cite{AC,ACL}). Our framework allows for the inclusion of a twisted term $\chi$, encompassing the case of scalar-flat metrics with cone singularities along divisors.
\end{itemize}

\medskip\noindent
\textbf{Main results.}
Our first main result establishes convexity of weighted twisted K-energy functionals along weak geodesics in the space of finite energy potentials.

\begin{mainthm}[Convexity of weighted twisted K-energy, see Theorem \ref{thm:main}]\label{thm:main0}
Let $(X, \omega)$ be a compact K\"ahler manifold, $T \subset \Aut_{red}(X,\chi)$ a compact torus, and $\chi$ a $T$-invariant positive $(1,1)$-current satisfying \eqref{eq: ChiProp}. For smooth weights $\mathrm{v},\mathrm{w}\in C^\infty(\mathfrak{t}^\vee)$ with $\mathrm{v}>0$ on the moment polytope $P_X$, the weighted twisted Mabuchi functional
\begin{equation}
\label{KEnergyDef}
\cM_{\mathrm{v},\mathrm{w}}^{\chi}(\varphi):=
\ent_{\mathrm{v}}(\varphi)
+
\enR_{\mathrm{v}}^{\chi}(\varphi)+\enE_{\mathrm{v}\,\mathrm{w}}(\varphi)
\end{equation}
admits a unique greatest lower semicontinuous extension $\cM_{\mathrm{v},\mathrm{w}}^{\chi}: \cET(X, \omega) \to \R\cup \{+\infty\}$, that is moreover convex and continuous along weak geodesics.
\end{mainthm}

Theorem \ref{thm:main} extends the finite-energy convexity theory of Berman--Darvas--Lu to the weighted twisted setting. More precisely, it generalizes the smooth weighted K-energy convexity framework of Apostolov--Jubert--Lahdili by allowing singular positive twists satisfying \eqref{eq: ChiProp} and by proving convexity on the finite-energy space $\cET(X,\omega)$. An important special case occurs when $\chi=2\pi[D]$ is the current of integration along a simple normal crossing divisor $D=\sum_{i=1}^k D_i + \sum_{j=k+1}^N (1-\alpha_j)D_j$ where the first $k$ components have cusp singularities (cone angle $0$) and the remaining components have conic singularities (cone angle $2\pi\alpha_j$ for $\alpha_j \in (0,1)$). In the mixed cusp/conic case, the theorem gives a finite-energy variational extension of the Poincar\'e-type framework developed by Auvray \cite{Auvray13,Auvray,Auvray1,Auvray18} and of the toric extremal Poincar\'e analysis of Apostolov--Auvray--Sektnan \cite{ApostolovAuvraySektnan}.

\begin{cor}[Weighted K-energy convexity for mixed cusp and conic singularities, see Corollary \ref{prop:extention-K-energy-divisor}]\label{cor:mixed-singularities}
The weighted twisted Mabuchi functional $\cM_{\mathrm{v},\mathrm{w}}^{2\pi[D]}$ for mixed cusp/conic divisors admits a unique greatest lower semicontinuous extension to $\cET(X, \omega)$, that is moreover convex and continuous along weak geodesics.
\end{cor}

Regarding Corollary \ref{cor:mixed-singularities}, Xu \cite{Xu25} and Xu--Zheng \cite{Xu_Zheng25} established uniqueness results for extremal and cscK Poincar\'e-type metrics along smooth divisors, employing the convexity framework of Berman--Berndtsson \cite{BB} where applying a partition of unity involves subtleties near the divisor. Thus Corollary \ref{cor:mixed-singularities} can be viewed as a finite-energy generalization of the smooth-divisor convexity framework appearing in \cite{Xu25,Xu_Zheng25}. It replaces the smooth-divisor and Poincar\'e-type assumptions by the general current decomposition \eqref{eq: ChiProp}, thereby allowing simple normal crossing divisors with mixed cusp and conic components and working on the full finite-energy space $\cET(X,\omega)$. Furthermore, we establish in Proposition \ref{prop:first-var-weak-poincare} the first variation formula for the weighted twisted Mabuchi functional $\cM_{\mathrm{v},\mathrm{w}}^{2\pi[D]}$ on the space of Poincar\'e-type potentials, showing that critical points satisfy the weighted twisted scalar curvature equation~\mbox{$S_{\mathrm{v}}^{2\pi[D]}(\omega_\varphi) = \mathrm{w}$}. A key ingredient in the proof is the following decomposition formula for the Ricci curvature of a Poincar\'e-type K\"ahler metric $\omega_\varphi$ established in Corollary \ref{decomp}:
\[
\Ric(\omega_\varphi) = \theta_\varphi +\theta'_\varphi+ 2\pi[D],
\]
where $\theta_\varphi$ is smooth on $X\setminus D$ and controlled near $D$, and $\theta'_\varphi$ is controlled by a Poincar\'e-type metric.

\subsubsection*{Openness of coercivity}
We now turn to the openness of the coercivity condition. A natural question, repeatedly arising in K\"ahler geometry, is to understand the stability of
cscK metrics---possibly with prescribed singularities---under variations of the underlying data,
such as small deformations of the complex structure or perturbations of the K\"ahler (or twisting)
class.  This viewpoint goes back to the deformation theory of LeBrun--Simanca
and has since been developed in many directions, including variational and analytic approaches and
gluing constructions; see for instance \cite{LeBrunSimanca94,Tosatti10,Donaldson2012,KZ, PTT23, BJT, PT26}
and the references therein.  In the present setting, we prove a finite-energy openness theorem for the relative weighted twisted K-energy under perturbations of the singular twist class.

It is natural to work with the \emph{relative} (i.e.\ $T^\C$-invariant) version of the weighted twisted K-energy. Indeed, as in the classical cscK problem, the vanishing of the (weighted twisted) Futaki character is a necessary condition for the existence of critical points. The relative theory uses Calabi's extremal modification: one introduces the weighted twisted extremal function $\ell^{\mathrm{ext}}_{\omega,\mathrm{v},\mathrm{w},\chi}$ and the associated relative functional
\[
\cM^\mathrm{rel}_{\mathrm{v},\mathrm{w},\chi}:=\cM_{\mathrm{v},\mathrm{w}\ell^{\mathrm{ext}}_{\omega,\mathrm{v},\mathrm{w},\chi}}^{\chi},
\]
so that the corresponding Futaki obstruction vanishes and $\cM^\mathrm{rel}_{\mathrm{v},\mathrm{w},\chi}$ is $T^\C$-invariant. We refer to Subsection~\ref{subsec:futaki} for the precise definitions and further discussion.
\begin{mainthm}[Openness of coercivity for mixed cusp and conic singularities, see Theorem \ref{thm:coeropen0}]\label{thm:coeropen}
Fix smooth weights $\mathrm{v},\mathrm{w}\in C^\infty(\ft^\vee)$ with $\mathrm{v},\mathrm{w}>0$ on the (normalized) moment polytope $P$.
Let $D=\sum_{i=1}^k D_i + \sum_{j=k+1}^N (1-\alpha_j)D_j$ be a simple normal crossing divisor, and consider the twist class $\chi = 2\pi[D]$.

Assume that there exist $\delta > 0, A \in \R$ such that 
$$
	\cM^\mathrm{rel}_{\mathrm{v},\mathrm{w},\chi}(\varphi) \geq \delta \, d_{1, T^\mathbb{C}}(\varphi,0) - A
	$$
	for all $\varphi \in \cET(X, \omega)$ with $\enE(\f)= 0$. Then, for any $\hat{\delta} < \delta$, there exist $\hat{A} \in \R$ and a neighborhood $\mathcal{U}$ of $\boldsymbol{\alpha}=(0,\dots,0,\alpha_{k+1},\dots,\alpha_N)\in [0,1]^N$ such that 
	$$
	\cM^\mathrm{rel}_{\mathrm{v},\mathrm{w},\hat{\chi}}(\varphi) \geq \hat{\delta} \, d_{1, T^\mathbb{C}}(\varphi,0) - \hat{A}
	$$
	for all $\varphi \in \cET(X, \omega)$ with $\enE(\f)= 0$ and any $\hat{\boldsymbol{\alpha}}=(\hat{\alpha}_{1},\dots,\hat{\alpha}_N) \in \mathcal{U}$, where $\hat{\chi}=\sum_{j=1}^N 2\pi(1-\hat{\alpha}_j)[D_j]$. Moreover, $\delta-\hat \delta$ is linearly controlled by $\boldsymbol{\alpha} - \hat{\boldsymbol{\alpha}}$.

In particular, if $\cM^\mathrm{rel}_{\mathrm{v},\mathrm{w},\chi}$ is coercive relative to $T^\mathbb{C}$, then $\cM^\mathrm{rel}_{\mathrm{v},\mathrm{w},\hat{\chi}}$ is coercive relative to $T^\mathbb{C}$ for any $\hat{\boldsymbol{\alpha}} \in \mathcal{U}$. 
\end{mainthm}

In \cite{aoi2022}, the author discusses that the existence
of a Poincar\'e-type cscK metric implies the existence of cscK cone metrics with small cone angles. The existence of a Poincar\'e-type cscK metric along a smooth hypersurface $D\in |L_X|$ in a polarized manifold $(X,L_X)$ is used to obtain cscK cone metrics in the fixed class $c_1(L_X)$ with sufficiently small cone angles. Aoi's approach is analytic and relies on a deformation argument, and therefore requires additional assumptions, including the triviality of the relevant automorphism groups of both $D$ and the pair $(X,D)$. By contrast, the result proved here is purely variational: starting from coercivity of the finite-energy twisted K-energy at the cusp, it gives coercivity of the corresponding conic K-energy for all sufficiently small positive cone angles. Moreover, the argument applies to simple normal crossing divisors and to the relative setting with holomorphic vector fields.

Sz\'ekelyhidi \cite{sze06} formulated a notion of relative K-stability for the triple $(X, D, L)$ in the polarized case, utilizing a generalized Donaldson--Futaki invariant for test configurations. Crucially, he identified a specific numerical constraint---related to the deformation to the normal cone of $D \subset X$---necessary to ensure the metric exhibits Poincar\'e asymptotics. On the differential-geometric side, Auvray \cite{Auvray13, Auvray18} established that the existence of a Poincar\'e-type cscK (resp. extremal) metric implies a slope inequality relating the average scalar curvature (resp. extremal function) of $[\omega]$ to that of $[\omega]|_{D}$, which corresponds precisely to Sz\'ekelyhidi's numerical constraint. Furthermore, in the cscK (resp. extremal) setting, \cite{Auvray1} shows that existence requires each component $D_j$ to admit a smooth cscK (resp. extremal) metric in the class $[\omega]|_{D_j}$. In the toric setting, Apostolov, Auvray, and Sektnan \cite{ApostolovAuvraySektnan} demonstrated that the existence of an extremal Poincar\'e-type K\"ahler metric indeed implies a corresponding notion of K-stability. Inspired by \cite[Conjecture 4.14]{ApostolovAuvraySektnan} and the stability framework established by Sz\'ekelyhidi, we propose the following conjecture.

For simplicity, we assume that $D$ is a smooth divisor. Let $\ell^{\mathrm{ext}} := \ell^{\mathrm{ext}}_{\omega,1,1,2\pi[D]}$ denote the twisted extremal function on $X$ associated to the twist $\chi = 2\pi[D]$ (as defined in Section \ref{subsec:futaki}), and let $\ell^{\mathrm{ext}}_{D} := \ell^{\mathrm{ext}}_{\omega|_{D},1,1}$ denote the (untwisted) extremal function intrinsic to the triple $(D, [\omega]|_{D}, T)$.
A $(1,1,2\pi[D])$-extremal Poincar\'e-type K\"ahler metric is defined by a $T$-invariant Poincar\'e-type potential $\varphi$ such that $S^{2\pi[D]}(\omega_\varphi) = \ell^{\mathrm{ext}}$. This equation is a special case of \eqref{eq:wtscK} with $\mathrm{v}=1$ and $\mathrm{w} = \ell^{\mathrm{ext}}$.

\begin{conj}[{Existence of $(1,1,2\pi[D])$-extremal Poincar\'e-type K\"ahler metrics}]\label{conj:twisted_extremal}
Let $(X, \omega)$ be a compact K\"ahler manifold and let $D$ be a smooth divisor. Let $T$ be a maximal compact torus in the reduced automorphism group $\Aut_{red}(X, D)$. The manifold $X$ admits a $T$-invariant $(1,1,2\pi[D])$-extremal metric in $[\omega]$ if and only if the following conditions are satisfied:

\begin{itemize}
    \item[(i)] \emph{Global Coercivity:} The twisted K-energy $\cM_{1, \ell^{\mathrm{ext}}}^{2\pi[D]}$ is coercive on the space of finite energy potentials $\cET(X, \omega)$ relative to $T^\mathbb{C}$;
    
    \item[(ii)] \emph{Boundary Coercivity:} The extremal K-energy $\cM_{1, \ell^{\mathrm{ext}}_{D}}$ is coercive on the space of finite energy potentials $\cET(D, \omega|_{D})$ relative to $T^\mathbb{C}$;
    
    \item[(iii)] \emph{Boundary Stability Condition:} On the divisor $D$, the extremal functions satisfy:
    \[
    \ell^{\mathrm{ext}}_{D} - \ell^{\mathrm{ext}} \big|_{D} = c > 0
    \]
    for some constant $c$.
\end{itemize}
\end{conj}

\begin{remark}
When $D = \sum_{i=1}^k D_i$ is a simple normal crossing divisor, the conditions above must be generalized to all strata $D_I = \bigcap_{i \in I} D_i$. Specifically:
\begin{itemize}
    \item Condition (ii) requires coercivity of the twisted K-energy corresponding to the pair $(D_I, D_I \cap (\bigcup_{j \notin I} D_j))$.
\item Condition (iii) requires the positivity of the difference between the extremal function intrinsic to the pair $(D_I, D_I \cap (\bigcup_{j \notin I} D_j))$ and the restriction of the extremal function of $(D_I, \omega|_{D_I})$ to $D_I \cap (\bigcup_{j \notin I} D_j)$.\end{itemize}
\end{remark}

\subsection*{Organization of the paper.} Section~2 develops the weighted twisted formalism. We recall the moment-map normalization, the smooth weighted Monge--Amp\`ere operators and energy functionals, and then define the finite-energy extension of the weighted twisted Mabuchi functional using the decomposition of the twisting current. We also fix the Futaki invariant and the extremal normalization needed for the relative theory.

Section~3 proves the convexity theorem. After recalling weak geodesics, we establish the convexity formula first in the smooth setting and then pass to the finite-energy extension by approximation. Section~4 applies this result to the relative weighted Mabuchi functional and proves openness of coercivity under perturbations of cone angles, including conic K-energy coercivity and the passage from cusp coercivity to small-angle conic coercivity.

The appendix relates the abstract finite-energy construction to Poincar\'e type geometry. We first show that Poincar\'e type potentials lie in the finite-energy domain. We then prove the Ricci-current decomposition for Poincar\'e type metrics and use it to compute the first variation of the associated extended weighted twisted Mabuchi functional on the Poincar\'e type potentials. Thus the functional studied in the main body recovers the expected logarithmic Mabuchi functional on Poincar\'e type metrics.

\tableofcontents

\section{Weighted twisted formalism}\label{sec:wen}

We fix in this section the notation and variational framework used throughout the paper. 
The weighted scalar curvature, weighted Futaki invariant, and weighted Mabuchi formalism were introduced and developed by Lahdili in \cite{Lah,Lah23}. 
Their coercivity and finite-energy aspects were subsequently developed in \cite{AJL}, while the equivariant weighted Monge--Amp\`ere formalism and the precise extension results used below are presented systematically in \cite{BJT}. 
The underlying unweighted pluripotential theory is due to Berman--Darvas--Lu \cite{BDL}. 
We recall these constructions in the form needed for the weighted twisted setting.

\subsection{Moment maps and finite-energy potentials}

We use the standard Hamiltonian conventions for torus actions and momentum images from \cite{Atiyah,GS}. 
For finite-energy potentials and weak geodesics, we use the $d_1$-geometry developed in \cite{c,BK,Dar,Darvas_Rubinstein_2017,CTW18,BDL}, together with the equivariant conventions adopted in the weighted setting in \cite{AJL,BJT}.

Let $(X,\om)$ be a compact K\"ahler manifold of dimension $n$. 
Throughout the paper we use the convention
\[
\dc=i(\dbar-\partial),
\qquad
\ddc=2i\partial\dbar .
\]

Let $\chi$ be a fixed closed real $T$-invariant $(1,1)$-current on $X$, and fix a compact torus
\[
T\subset\Aut_{red}(X,\chi)
:=
\{g\in\Aut_{red}(X)\mid g^*\chi=\chi\},
\]
with Lie algebra $\ft$. 
Here $\Aut_{red}(X)$ denotes the reduced automorphism group, as in the standard references on holomorphic vector fields with zeros; see for instance \cite{gauduchon-book,Fuj}. 
We denote by $\cH^T(X,\om)$ the space of smooth $T$-invariant K\"ahler potentials in the class $[\om]$.

For a closed real $T$-invariant $(1,1)$-form $\theta$, a moment map is a $T$-invariant map
\[
m_\theta:X\to\ft^\vee
\]
such that, for each $\xi\in\ft$,
\begin{equation}\label{equ:moment}
-\de m_\theta^\xi
=
i(\xi)\theta
=
\theta(\xi,\cdot),
\qquad
m_\theta^\xi:=\langle m_\theta,\xi\rangle .
\end{equation}
The moment map is unique up to an additive constant in $\ft^\vee$. 
We use the same convention for closed $T$-invariant $(1,1)$-currents, in which case the moment map is understood distributionally. 
In particular, if $u$ is a $T$-invariant distribution, then
\begin{equation}\label{equ:mddc}
m_u^\xi
:=
\dc u(\xi)
=
-\de u(J\xi)
\end{equation}
is the moment map of the current $\ddc u$.

For $\f\in\cH^T(X,\om)$, we write
\[
\om_\f:=\om+\ddc\f,
\qquad
m_{\om_\f}:=m_\om+m_\f .
\]
We choose the additive normalization of $m_\om$ once and for all so that the momentum image
\[
P_X:=m_{\om_\f}(X)\subset\ft^\vee
\]
is independent of $\f$. 
After this normalization, all weights will be regarded as smooth functions on a neighborhood of the compact polytope $P_X$.

We also use the finite-energy space
\[
\cET(X,\om):=\mathcal E^1(X,\om)^T,
\]
the closed subspace of $T$-invariant elements in the usual finite-energy space. 
Equivalently, $\cET(X,\om)$ is the $d_1$-metric completion of $\cH^T(X,\om)$. 
All variational formulae below are first understood on $\cH^T(X,\om)$, and their finite-energy extensions are recorded afterward.

\subsection{Smooth weighted Monge--Amp\`ere operators}

The weighted scalar-curvature framework originates in Lahdili's work \cite{Lah,Lah23}. 
We use its equivariant weighted Monge--Amp\`ere formulation developed in \cite{AJL,BJT}; in particular, the definitions and integration-by-parts identities below are those of \cite[Definition~3.9 and Lemma~3.10]{BJT}.

Let $\mathrm{v}\in C^\infty(\ft^\vee)$ be a smooth weight. 
The $\mathrm{v}$-weighted Monge--Amp\`ere measure of $\f\in\cH^T(X,\om)$ is
\begin{equation}\label{equ:MAv}
\MA_{\mathrm{v}}(\f)
:=
\mathrm{v}(m_{\om_\f})\,\om_\f^n .
\end{equation}
If $\chi$ is a closed $T$-invariant $(1,1)$-current with moment map $m_\chi$, the corresponding twisted weighted operator is
\begin{equation}\label{equ:twMA}
\MA_{\mathrm{v}}^\chi(\f)
:=
\mathrm{v}(m_{\om_\f})\,n\,\chi\wedge\om_\f^{n-1}
+
\langle \mathrm{v}'(m_{\om_\f}),m_\chi\rangle\,\om_\f^n .
\end{equation}

Dividing by $\MA_{\mathrm{v}}(\f)$ gives the weighted trace. 
At $\f=0$, this reads
\[
\tr_{\om,\mathrm{v}}(\chi)
:=
\frac{\MA_{\mathrm{v}}^\chi(0)}{\MA_{\mathrm{v}}(0)}
=
\tr_\om(\chi)
+
\langle(\log \mathrm{v})'(m_\om),m_\chi\rangle .
\]
More generally, for smooth $\f$,
\[
\tr_{\om_\f,\mathrm{v}}(\chi)
:=
\frac{\MA_{\mathrm{v}}^\chi(\f)}{\MA_{\mathrm{v}}(\f)}.
\]

The basic integration-by-parts identity is
\begin{equation}\label{equ:twMAddc}
\int_X g\,\MA_{\mathrm{v}}^{\ddc f}(\f)
=
-n\int_X
\mathrm{v}(m_{\om_\f})\,\de g\wedge\dc f\wedge\om_\f^{n-1}.
\end{equation}
In particular, for $T$-invariant smooth functions $f,g$,
\begin{equation}\label{equ:twMAsymm}
\int_X g\,\MA_{\mathrm{v}}^{\ddc f}(\f)
=
\int_X f\,\MA_{\mathrm{v}}^{\ddc g}(\f).
\end{equation}
This symmetry is the variational closedness property behind the weighted energy functionals.

\subsection{Smooth weighted energies}

The weighted energy and Mabuchi formalism was introduced by Lahdili in \cite{Lah,Lah23} and further developed in \cite{AJL,BJT}. 
We use the twisted weighted energy construction of \cite[Proposition~3.13 and Lemma~3.14]{BJT}.

The $\mathrm{v}$-weighted Monge--Amp\`ere energy is denoted by $\enE_{\mathrm{v}}$ and is characterized by
\begin{equation}\label{equ:Ev-first-var}
\frac{d}{dt}\bigg|_{t=0}
\enE_{\mathrm{v}}(\f+t\psi)
=
\int_X \psi\,\MA_{\mathrm{v}}(\f).
\end{equation}

Similarly, for a closed $T$-invariant $(1,1)$-current $\chi$, we denote by
\[
\enE_{\mathrm{v}}^\chi:\cH^T(X,\om)\to\R
\]
a variational primitive of $\MA_{\mathrm{v}}^\chi$, so that
\begin{equation}\label{equ:Evchi-first-var}
\frac{d}{dt}\bigg|_{t=0}
\enE_{\mathrm{v}}^\chi(\f+t\psi)
=
\int_X \psi\,\MA_{\mathrm{v}}^\chi(\f).
\end{equation}
The special case of exact twists is particularly useful. 
If $\chi=\ddc u$, with $u$ a $T$-invariant distribution for which the expression is defined, then one may take
\begin{equation}\label{eq:energy-ddc}
\enE_{\mathrm{v}}^{\ddc u}(\f)
=
\int_X u\,\MA_{\mathrm{v}}(\f).
\end{equation}
Consequently, if $\chi=\chi_0+\ddc u$ with $\chi_0$ smooth, then, up to an additive constant,
\begin{equation}\label{eq:energy-general-current}
\enE_{\mathrm{v}}^\chi(\f)
=
\enE_{\mathrm{v}}^{\chi_0}(\f)
+
\int_X u\,\MA_{\mathrm{v}}(\f).
\end{equation}

\subsection{Weighted twisted scalar curvature and smooth Mabuchi energy}

The weighted scalar curvature, weighted Mabuchi functional, and weighted extremal equation were introduced by Lahdili in \cite{Lah,Lah23}; see also \cite{AJL}. 
We use the equivalent equivariant formulation of \cite[Sections~3.4--3.6]{BJT} and incorporate the closed $T$-invariant twist $\chi$. 
In particular, the first-variation formula below is the twisted counterpart of \cite[Theorem~5]{Lah} and \cite[Proposition~3.30]{BJT}.

\begin{defi}\label{def:weighted-twisted-scalar-curvature}\label{def:equivariant-ricci}
For $\varphi\in\cH^T(X,\omega)$, the $\mathrm{v}$-weighted Ricci curvature is
\begin{equation}\label{equ:Ric-v}
\Ric_{\mathrm{v}}(\omega_\varphi)
:=
-\frac12\ddc
\log\bigl(\mathrm{v}(m_{\omega_\varphi})\omega_\varphi^n\bigr).
\end{equation}
If $\chi$ is a closed $T$-invariant $(1,1)$-current, the weighted twisted scalar
curvature is defined by
\begin{equation}\label{equ:Svchi-measure}
S_{\mathrm{v}}^\chi(\omega_\varphi)\MA_{\mathrm{v}}(\varphi)
:=
\MA_{\mathrm{v}}^{\Ric_{\mathrm{v}}(\omega_\varphi)-\chi}(\varphi).
\end{equation}
Equivalently, using the weighted trace notation,
\begin{equation}\label{equ:Svchi}
S_{\mathrm{v}}^\chi(\omega_\varphi)
=
\tr_{\omega_\varphi,\mathrm{v}}
\bigl(\Ric_{\mathrm{v}}(\omega_\varphi)-\chi\bigr).
\end{equation}
The weighted twisted scalar curvature equation is
\begin{equation}\label{equ:wtscK}
S_{\mathrm{v}}^\chi(\omega_\varphi)
=
\mathrm{w}(m_{\omega_\varphi}).
\end{equation}
\end{defi}

When $\mathrm{v}\equiv1$, the left-hand side becomes the usual twisted scalar
curvature
\[
S^\chi(\omega_\varphi)
=
\tr_{\omega_\varphi}
\bigl(\Ric(\omega_\varphi)-\chi\bigr).
\]
Its average is the topological constant
\[
\bar S^\chi
:=
\frac{
n\int_X(\Ric(\omega)-\chi)\wedge\omega^{n-1}
}{
\int_X\omega^n
}.
\]
Thus, when $\mathrm{v}\equiv1$ and $\mathrm{w}=\bar S^\chi$, equation
\eqref{equ:wtscK} reduces to the usual twisted cscK equation. If moreover
$\chi=0$, this recovers the ordinary cscK equation with
$\mathrm{w}=\bar S$.

\begin{defi}\label{defi:relative-entropy}
For positive Radon measures $\mu,\nu$ on $X$, we use the entropy convention
\begin{equation}\label{equ:relative-entropy}
\Ent(\mu\mid\nu)
=
\begin{cases}
\displaystyle\int_X h\log h\,d\nu,
&
\text{if } \mu=h\nu,\\[0.4em]
+\infty,
&
\text{otherwise.}
\end{cases}
\end{equation}
\end{defi}

The $\mathrm{v}$-weighted entropy is
\begin{equation}\label{equ:ent-v}
\ent_{\mathrm{v}}(\f)
:=
\frac{1}{2}\Ent\bigl(\MA_{\mathrm{v}}(\f)\mid\om^n\bigr),
\end{equation}
and the $\mathrm{v}$-weighted Ricci energy twisted by $\chi$ is
\begin{equation}\label{equ:Ricci-energy-vchi}
\enR_{\mathrm{v}}^\chi(\f)
:=
\enE_{\mathrm{v}}^{-(\Ric(\om)-\chi)}(\f).
\end{equation}
The factor $1/2$ in the definition of $\Ric_{\mathrm{v}}$ is compatible with the decomposition condition \eqref{eq: ChiProp} in the introduction.

Let $\mathrm{w}\in C^\infty(\ft^\vee)$. 
We denote by $\enE_{\mathrm{v}\,\mathrm{w}}$ the weighted energy corresponding to the product weight $\mathrm{v}\,\mathrm{w}$. 
Thus
\begin{equation}\label{equ:Evw-first-var}
\frac{d}{dt}\bigg|_{t=0}
\enE_{\mathrm{v}\,\mathrm{w}}(\f+t\psi)
=
\int_X
\psi\,\mathrm{w}(m_{\om_\f})\,\MA_{\mathrm{v}}(\f).
\end{equation}

The weighted Mabuchi energy twisted by $\chi$ is then
\begin{equation}\label{eq:weighted-twisted-mabuchi}
\cM_{\mathrm{v},\mathrm{w}}^\chi(\f)
:=
\ent_{\mathrm{v}}(\f)
+
\enR_{\mathrm{v}}^\chi(\f)
+
\enE_{\mathrm{v}\,\mathrm{w}}(\f).
\end{equation}
Its first variation is
\begin{equation}\label{equ:twisted-mabuchi-first-var}
\frac{d}{dt}\bigg|_{t=0}
\cM_{\mathrm{v},\mathrm{w}}^\chi(\f+t\psi)
=
\int_X
\psi\,
\bigl(
\mathrm{w}(m_{\om_\f})-S_{\mathrm{v}}^\chi(\om_\f)
\bigr)
\MA_{\mathrm{v}}(\f).
\end{equation}
Thus the smooth critical points of $\cM_{\mathrm{v},\mathrm{w}}^\chi$ are precisely the solutions of
\begin{equation}
S_{\mathrm{v}}^\chi(\om_\f)
=
\mathrm{w}(m_{\om_\f}).
\end{equation}
The formula \eqref{equ:twisted-mabuchi-first-var} is a smooth variational
identity. For finite-energy potentials, and in particular for Poincar\'e-type
potentials, the same first-variation formula is not automatic from the lower
semicontinuous extension. In Subsection~\ref{subsec:log-mabuchi-poincare}, the
extended functional is identified with the variational primitive of the
Poincar\'e-type scalar curvature operator.

\subsection{Extension to finite energy}\label{subsec:finite-energy-extension}

We now extend the smooth weighted twisted formalism to the finite-energy space
\[
\cET(X,\om):=\mathcal E^1(X,\om)^T .
\]
The unweighted finite-energy framework used here is due to Berman--Darvas--Lu \cite{BDL}. 
The corresponding weighted Monge--Amp\`ere operators, energies, and continuity estimates were developed in \cite{AJL,BJT}; the precise measure and energy extension statements used below are collected in \cite[Propositions~3.39--3.40]{BJT}.

The weighted Monge--Amp\`ere operator admits a unique continuous extension
\[
\cET(X,\om)\ni\f
\longmapsto
\MA_{\mathrm{v}}(\f)
\]
with values in the space of $T$-invariant signed Radon measures. 
If $\mathrm{v}\geq 0$, then $\MA_{\mathrm{v}}(\f)$ is a positive Radon measure of total mass
\[
V_{\mathrm{v}}:=\int_X \mathrm{v}(m_\om)\om^n .
\]
Moreover, $\MA_{\mathrm{v}}(\f)$ integrates all finite-energy potentials, and for all $\f,\p,\tau\in\cET(X,\om)$ one has
\begin{equation}\label{equ:MAvhold}
\left|
\int_X \tau\,\bigl(\MA_{\mathrm{v}}(\f)-\MA_{\mathrm{v}}(\p)\bigr)
\right|
\leq
C\|\mathrm{v}\|_{C^0(P_X)}
d_1(\f,\p)^{1/2}
\max\{d_1(\f,0),d_1(\p,0),d_1(\tau,0)\}^{1/2},
\end{equation}
where $C$ depends only on $n=\dim_\C X$.

The weighted energy $\enE_{\mathrm{v}}$ also extends continuously to $\cET(X,\om)$, with
\begin{equation}\label{equ:wenlip}
\left|
\enE_{\mathrm{v}}(\f)-\enE_{\mathrm{v}}(\p)
\right|
\leq
C\|\mathrm{v}\|_{C^0(P_X)}\,d_1(\f,\p).
\end{equation}
If $\theta$ is a smooth closed $T$-invariant $(1,1)$-form, then $\enE_{\mathrm{v}}^\theta$ also extends continuously to $\cET(X,\om)$, and satisfies an estimate of the form
\begin{equation}\label{equ:twenhold}
\left|
\enE_{\mathrm{v}}^\theta(\f)-\enE_{\mathrm{v}}^\theta(\p)
\right|
\leq
C(\mathrm{v},\theta)\,
d_1(\f,\p)^\alpha
\max\{d_1(\f,0),d_1(\p,0)\}^{1-\alpha},
\qquad
\alpha:=2^{-n}.
\end{equation}

We shall also use exact twists whose potentials are only of finite energy. 
If $u\in\cET(X,\om)$, then \eqref{equ:MAvhold} allows us to define
\begin{equation}\label{eq:energy-ddc-finite}
\enE_{\mathrm{v}}^{\ddc u}(\f)
:=
\int_X u\,\MA_{\mathrm{v}}(\f),
\qquad
\f\in\cET(X,\om).
\end{equation}
Indeed, \eqref{equ:MAvhold} gives
\begin{equation}\label{equ:energy-ddc-finite-continuity}
\left|
\enE_{\mathrm{v}}^{\ddc u}(\f)-\enE_{\mathrm{v}}^{\ddc u}(\p)
\right|
\leq
C\|\mathrm{v}\|_{C^0(P_X)}
d_1(\f,\p)^{1/2}
\max\{d_1(\f,0),d_1(\p,0),d_1(u,0)\}^{1/2}.
\end{equation}
Thus $\enE_{\mathrm{v}}^{\ddc u}$ is continuous on $\cET(X,\om)$.

We now impose the decomposition condition on the twisting current:
\begin{equation}\label{eq:chi-decomposition}
\chi
=
\beta+\frac{1}{2}\ddc f+\ddc\hat f,
\end{equation}
where $\beta$ is smooth, $f$ is upper semicontinuous with $e^{-f}\in L^1(X,\om^n)$, and $\hat f\in\cET(X,\om)$, as in \eqref{eq: ChiProp}. 
With the decomposition \eqref{eq:chi-decomposition}, the singular term $\frac{1}{2}\ddc f$ is incorporated into the reference measure of the entropy.
We define
\begin{equation}\label{equ:ent-v-f}
\ent_{\mathrm{v}}^f(\f)
:=
\frac{1}{2}
\Ent\bigl(\MA_{\mathrm{v}}(\f)\mid e^{-f}\om^n\bigr),
\qquad
\f\in\cET(X,\om).
\end{equation}
This functional is lower semicontinuous and may take the value $+\infty$.

On smooth potentials, and whenever the terms are finite, this definition is compatible with the usual smooth decomposition. Indeed,
\begin{equation}\label{equ:entropy-shift}
\ent_{\mathrm{v}}^f(\f)
=
\ent_{\mathrm{v}}(\f)
+
\frac{1}{2}\int_X f\,\MA_{\mathrm{v}}(\f)
=
\ent_{\mathrm{v}}(\f)
+
\enE_{\mathrm{v}}^{\frac{1}{2}\ddc f}(\f).
\end{equation}
Thus the finite-energy definition agrees with the smooth formula after the exact singular contribution $\frac{1}{2}\ddc f$ is absorbed into the entropy.

The remaining part of the Ricci energy is continuous. 
We define
\begin{equation}\label{equ:Ricci-energy-decomposed}
\enR_{\mathrm{v}}^{\beta+\ddc\hat f}(\f)
:=
\enE_{\mathrm{v}}^{-\Ric(\om)+\beta}(\f)
+
\int_X \hat f\,\MA_{\mathrm{v}}(\f).
\end{equation}
The first term is continuous by \eqref{equ:twenhold}, applied to the smooth form
$\theta=-\Ric(\om)+\beta$. The second term is
$\enE_{\mathrm{v}}^{\ddc\hat f}$, hence is continuous by
\eqref{equ:energy-ddc-finite-continuity}, applied with $u=\hat f$.
Thus $\enR_{\mathrm{v}}^{\beta+\ddc\hat f}$ is continuous on $\cET(X,\om)$.

With the decomposition \eqref{eq:chi-decomposition}, the finite-energy representative of the weighted twisted Mabuchi functional is
\begin{equation}\label{eq:Mabuchi-finite-energy-decomposed}
\cM_{\mathrm{v},\mathrm{w}}^\chi(\f)
:=
\ent_{\mathrm{v}}^f(\f)
+
\enR_{\mathrm{v}}^{\beta+\ddc\hat f}(\f)
+
\enE_{\mathrm{v}\,\mathrm{w}}(\f),
\qquad
\f\in\cET(X,\om).
\end{equation}
By the compatibility above, on $\cH^T(X,\om)$,
\begin{align}
\cM_{\mathrm{v},\mathrm{w}}^\chi(\f)
&=
\ent_{\mathrm{v}}^f(\f)
+
\enR_{\mathrm{v}}^{\beta+\ddc\hat f}(\f)
+
\enE_{\mathrm{v}\,\mathrm{w}}(\f)
\label{eq:form1}\\
&=
\ent_{\mathrm{v}}(\f)
+
\enR_{\mathrm{v}}^\chi(\f)
+
\enE_{\mathrm{v}\,\mathrm{w}}(\f).
\label{eq:form2}
\end{align}

We need the following weighted version of the Berman--Darvas--Lu entropy approximation theorem, allowing the singular reference measure $e^{-f}\omega^n$ and using the solvability and stability theory for the weighted Monge--Amp\`ere equation.

\begin{prop}\label{Ent-approx}
Assume that $f$ is upper semicontinuous on $X$ and that $e^{-f}\in L^1(X,\om^n)$. 
If $\mathrm{v}>0$, then the functional
\[
\cET(X,\om)\ni\f
\longmapsto
\Ent\bigl(\MA_{\mathrm{v}}(\f)\mid e^{-f}\om^n\bigr)
\]
is $d_1$-lower semicontinuous. 
Moreover, for every $\f\in\cET(X,\om)$, there exists a sequence $\f_j\in\cH^T(X,\om)$ such that
\[
d_1(\f_j,\f)\to0
\]
and
\[
\Ent\bigl(\MA_{\mathrm{v}}(\f_j)\mid e^{-f}\om^n\bigr)
\to
\Ent\bigl(\MA_{\mathrm{v}}(\f)\mid e^{-f}\om^n\bigr).
\]
\end{prop}

\begin{proof}
The argument is a weighted analogue of \cite[Lemma 3.1]{BDL}, with two additional points: the reference measure is the possibly singular measure $e^{-f}\omega^n$, and the recovery sequence is obtained by solving the weighted Monge--Amp\`ere equation. By the weighted Monge--Amp\`ere continuity estimate \eqref{equ:MAvhold} and the fact that the entropy $\mu\mapsto \Ent(\mu \mid e^{-f}\omega^n)$ is lsc on the space of finite measures, with respect to the weak convergence of measures, (cf. \cite[Prop. 3.1]{BB}), it follows that the entropy $\varphi\mapsto \Ent(\MA_{\mathrm{v}}(\varphi) \mid e^{-f}\omega^n)$ is $d_1$-lsc. Let $\varphi\in \cET(X, \omega)$. If $\Ent(\MA_{\mathrm{v}}(\varphi) \mid e^{-f}\omega^n)=+\infty$ then any sequence $\varphi_j\in\mathcal{H}^T(X,\omega)$ such that $d_1(\varphi_j,\varphi)\to 0$ satisfies $\Ent(\MA_{\mathrm{v}}(\varphi) \mid e^{-f}\omega^n)\to +\infty$ as $j\to \infty$. We suppose that $\Ent(\MA_{\mathrm{v}}(\varphi) \mid e^{-f}\omega^n)<+\infty$ and we put $g:=\frac{\MA_{\mathrm{v}}(\varphi)}{\omega^n}\geq 0$ the density function of the measure $\MA_{\mathrm{v}}(\varphi)$. As $e^{-f}\in L^1(X,\om^n),g \log\frac{g}{e^{-f}} \in L^1(X,\om^n)$, by the same construction as in \cite[Lemma 3.1]{BDL}, there exists a sequence of positive functions $g_j\in C^{\infty}_T(X)$ such that $\parallel g-g_j\parallel_{L^{1}}\to 0$ and
\[
\int_X g_j\log \frac{g_j}{e^{-f}}\, \omega^n\to \Ent(\MA_{\mathrm{v}}(\varphi) \mid e^{-f}\omega^n).
\]
Using \cite[Prop. 3.7]{HL}, we can find a smooth potential $\varphi_j\in \mathcal{H}^T(X,\omega)$ (which is unique up to adding a constant) such that $\MA_{\mathrm{v}}(\varphi_j)=\left(\frac{\int_X \mathrm{v}(m_\om)e^{-f}\omega^n}{\int_X g_j e^{-f}\omega^n } \right) g_j e^{-f}\omega^n$. By \cite[Lemma 2.16]{HL}, up to a passing to a subsequence of $\varphi_j$, there exists a $\psi\in \cET(X, \omega)$ such that $d_1(\psi,\varphi_j)\to 0$. The estimate \eqref{equ:MAvhold}, together with $\parallel g-g_j\parallel_{L^{1}}\to 0$, gives $\MA_{\mathrm{v}}(\psi)=\underset{j\to \infty}{\lim} \MA_{\mathrm{v}}(\varphi_j)=\MA_{\mathrm{v}}(\varphi)$. It follows that $\varphi=\psi$ (up to a constant) by \cite[Thm. 2.18]{BWN}. Thus, $d_1(\varphi,\varphi_j)\to 0$ and $\Ent(\MA_{\mathrm{v}}(\varphi_j) \mid e^{-f}\omega^n)\to \Ent(\MA_{\mathrm{v}}(\varphi) \mid e^{-f}\omega^n)$ as $j\to \infty$.
\end{proof}

Combining this entropy approximation with the continuity of the remaining energy terms gives the finite-energy extension of the weighted twisted Mabuchi functional.

\begin{prop}\label{prop:extension-mabuchi}
The weighted twisted Mabuchi functional admits a unique greatest lower semicontinuous extension
\[
\cM_{\mathrm{v},\mathrm{w}}^\chi:\cET(X,\om)\to\R\cup\{+\infty\}.
\]
This extension is given by the decomposed formula
\[
\cM_{\mathrm{v},\mathrm{w}}^\chi(\f)
=
\ent_{\mathrm{v}}^f(\f)
+
\enR_{\mathrm{v}}^{\beta+\ddc\hat f}(\f)
+
\enE_{\mathrm{v}\,\mathrm{w}}(\f).
\]
Moreover, for every $\f\in\cET(X,\om)$, there exists a sequence $\f_j\in\cH^T(X,\om)$ such that
\[
d_1(\f_j,\f)\to0
\]
and
\[
\cM_{\mathrm{v},\mathrm{w}}^\chi(\f_j)
\to
\cM_{\mathrm{v},\mathrm{w}}^\chi(\f).
\]
\end{prop}

\begin{remark}\label{rem:decomposed-extension}
The decomposed formula is essential for identifying the greatest lower semicontinuous extension. 
If one extends separately the terms in the smooth identity
\[
\cM_{\mathrm{v},\mathrm{w}}^\chi
=
\ent_{\mathrm{v}}+\enR_{\mathrm{v}}^\chi+\enE_{\mathrm{v}\,\mathrm{w}},
\]
then both $\ent_{\mathrm{v}}$ and the singular part of $\enR_{\mathrm{v}}^\chi$ may contribute lower semicontinuous terms. 
This separated extension need not coincide with the greatest lower semicontinuous extension of the smooth functional. 
The decomposition
\[
\chi=\beta+\frac{1}{2}\ddc f+\ddc\hat f
\]
is used precisely to absorb the singular contribution $\frac{1}{2}\ddc f$ into the reference measure of the entropy, leaving only continuous energy terms outside the entropy.
\end{remark}

\subsection{Futaki invariant and extremal normalization}\label{subsec:futaki}

The weighted Futaki invariant and weighted extremal normalization were introduced by Lahdili in \cite{Lah,Lah23}, extending the classical Calabi and Futaki--Mabuchi constructions \cite{calabi,FM}. 
The relative weighted Mabuchi functional and its coercivity modulo the complexified torus are developed in \cite{AJL,BJT}. 
We recall here the corresponding twisted normalization.

For an affine function $\ell=\xi+c\in\ft\oplus\R$ on $\ft^\vee$, set
\begin{equation}\label{equ:Futaki-twisted}
\mathcal F_{\mathrm{v},\mathrm{w}}^\chi(\ell)
:=
\int_X
\ell(m_\om)\mathrm{w}(m_\om)\MA_{\mathrm{v}}(0)
-
\int_X
\ell(m_\om)S_{\mathrm{v}}^\chi(\om)\MA_{\mathrm{v}}(0).
\end{equation}
If the equation $S_{\mathrm{v}}^\chi(\om_\f)=\mathrm{w}(m_{\om_\f})$ has a smooth solution, then
\[
\mathcal F_{\mathrm{v},\mathrm{w}}^\chi(\ell)=0
\]
for every affine function $\ell$.

When $\mathrm{w}>0$ on $P_X$, affine functions carry the positive definite pairing
\begin{equation}\label{equ:Futaki-pairing}
\langle \ell,\ell'\rangle_{\mathrm{w}}
:=
\int_X
\ell(m_\om)\ell'(m_\om)\,
\mathrm{w}(m_\om)\MA_{\mathrm{v}}(0)
=
\int_X
\ell(m_\om)\ell'(m_\om)\,
\MA_{\mathrm{v}\,\mathrm{w}}(0).
\end{equation}

The weighted twisted extremal affine function
\[
\ell^{\mathrm{ext}}
=
\ell^{\mathrm{ext}}_{\om,\mathrm{v},\mathrm{w},\chi}
\]
is defined as the $\mathrm{w}$-weighted projection of $S_{\mathrm{v}}^\chi(\om)$ onto the space of affine functions on the moment polytope; equivalently, it is the unique affine function satisfying
\begin{equation}\label{equ:ellext-twisted}
\langle \ell,\ell^{\mathrm{ext}}\rangle_{\mathrm{w}}
=
\int_X
\ell(m_\om)S_{\mathrm{v}}^\chi(\om)\MA_{\mathrm{v}}(0)
\end{equation}
for every affine function $\ell$.

With this choice, the Futaki character vanishes for the modified weight $\mathrm{w}\ell^{\mathrm{ext}}$. 
The corresponding relative weighted twisted Mabuchi functional is
\begin{equation}\label{equ:relative-mabuchi}
\cM^{\mathrm{rel}}_{\mathrm{v},\mathrm{w},\chi}
:=
\cM_{\mathrm{v},\mathrm{w}\ell^{\mathrm{ext}}}^{\chi}.
\end{equation}
Its smooth critical points are the $(\mathrm{v},\mathrm{w},\chi)$-extremal metrics.

We measure coercivity modulo the complexified torus action using the relative $d_1$ distance:
\begin{equation}\label{def:relative_dist}
d_{1,T^\mathbb{C}}(\f,0)
:=
\inf_{\sigma\in T^\mathbb{C}}
d_1(\sigma\cdot\f,0).
\end{equation}
\begin{remark}\label{rem:relative_dist_attained}
By \cite[Proposition~1.8]{BJT}, the infimum in \eqref{def:relative_dist} is attained.
\end{remark}
In the vanishing Futaki case, equivalently when $\ell^{\mathrm{ext}}=1$, the weighted twisted Mabuchi energy is said to be coercive relative to $T^\mathbb{C}$ if there exist constants $\delta,C>0$ such that
\begin{equation}\label{def:coercivity}
\cM_{\mathrm{v},\mathrm{w}}^\chi(\f)
\geq
\delta\,d_{1,T^\mathbb{C}}(\f,0)-C
\end{equation}
for all normalized potentials
\[
\f\in
\cET_0(X,\om)
:=
\{\psi\in\cET(X,\om)\mid \enE(\psi)=0\}.
\]

\section{{Convexity of weighted twisted K-energy }}

\subsection{Weak geodesics}
We now prove the convexity theorem. The argument extends the finite-energy convexity method of Berman--Darvas--Lu to the weighted twisted setting, using Lahdili's smooth weighted formalism and the finite-energy weighted Monge--Amp\`ere theory of \cite{AJL,BJT}. The main point is to combine the weighted geodesic convexity formula with the singular twist decomposition \eqref{eq: ChiProp}.

It was shown by Donaldson \cite{Donaldson-geod} and Semmes \cite{Semmes} that by letting $\tau:=e^{-t+is}$, the geodesic $(\f_t)_{t\in[0,1]}\in\cH^T(X, \omega)$ can be viewed as a smooth function $\Psi(x,\tau)$ on $\hat{X}:=X\times\mathbb{A}$, where $\mathbb{A}:=\{e^{-1}\leq|\tau|\leq 1\}$ is the corresponding annulus in $\mathbb{C}$, defined by
\begin{equation}\label{complex}
\Psi(x,\tau):=\f_{t}(x),
\end{equation}
which is invariant under the natural action of $\mathbb{G}:=\T\times\mathbb{S}^1$ on $\hat{X}$, and satisfies the following degenerate Monge-Amp\`ere equation on $\hat{X}$,
\begin{equation*}
\big(\pi_{X}^*\omega+dd^{c}\Psi\big)^{n+1}=0
\end{equation*}
where $\pi_{X}:\hat{X}\to X$ is the projection on the first factor. Hence, the problem of connecting two potentials $\f_0,\f_1\in\cH^T(X, \omega)$ by a geodesic $(\f_t)_{t\in[0,1]}\in\cH^T(X, \omega)$ is equivalent to finding a solution $\Psi\in C^\infty(\hat{X},\mathbb{R})^{\mathbb{G}}$ to the following boundary value problem
\begin{equation}\label{MA}
\begin{cases}
\big(\pi_{X}^*\omega+dd^{c}\Psi\big)^{n+1}=0,\\
\omega+dd^{c}\Psi_{|X_\tau}>0,\text{ for }\tau\in\mathbb{A},\\
\Psi(\cdot,e^{-1})=\f_1\text{ and }\Psi(\cdot,1)=\f_0.
\end{cases}
\end{equation}
where $X_\tau:=\pi_{\mathbb{A}}^{-1}(\tau)$ is a fiber of the projection $\pi_{\mathbb{A}}:\hat{X}\to \mathbb{A}$. 

The boundary value problem \eqref{MA} makes sense for $\mathbb{G}$-invariant bounded plurisubharmonic functions $\Psi\in \PSH(\hat{X},\pi_{X}^{\star}\omega)^{\mathbb{G}}\cap L^{\infty}$, using the Bedford--Taylor interpretation of $\big(\pi_{X}^*\omega+dd^{c}\Psi\big)^{n+1}$ as a Borel measure on $\hat{X}$. 

By a result of Chen \cite{c}, with complements of B\l ocki \cite{Blocki2013} and Chu--Tosatti--Weinkove \cite{CTW18}, the boundary value problem \eqref{MA} admits a unique $\mathbb{G}$-invariant solution $\Psi\in C^{1,1}(\hat{X},\R)$ such that $\pi_{X}^*\omega+dd^{c}\Psi$ is a positive current with bounded coefficients, up to the boundary, corresponding to a family of functions $(\f_t)_{t\in[0,1]}$ in the space $\cH^{1,1}(X,\omega)^T$ of all $T$-invariant functions $\f\in C^{1,1}(X,\R)$ such that $\omega_\f$ is a positive current with bounded coefficients. The curve $(\f_t)_{t\in[0,1]}\subset\cH^{1,1}(X,\omega)^T$ is called the \emph{weak geodesic segment} joining $\f_0,\f_1\in\cH^T(X, \omega)$.

In what follows, we use the notation $\omega^{[n]} := \frac{1}{n!}\omega^n$.

\begin{lemma}\label{lem-ddc-enE}
For a smooth family of $T$-invariant potentials $\f_t \in \cH^T(X, \omega)$ and a positive $T$-invariant current $\chi$, we have:
\[\ddc\enE_{\mathrm{v}}^{\chi}(\Psi)=\int_X \mathrm{v}(m_{\Psi})\pi^{*}_X\chi\wedge(\pi^{*}_X\omega+\ddc\Psi)^{[n]}+\int_X\langle \mathrm{v}'(m_{\Psi}),m_{\chi}\rangle\Big(\pi^{*}_X\omega+\ddc\Psi\Big)^{[n+1]}.\]
\end{lemma}

\begin{proof}
When $\chi$ is smooth, the formula follows from \cite{Lah}, Lemma 6.

For a general positive current $\chi$, we decompose it as $\chi = \chi_0 + \ddc f$ where $\chi_0$ is smooth and $f$ is $\chi$-PSH. By linearity, it suffices to establish the formula for the singular part $\ddc f$.
Since $f$ is upper semicontinous, there exists $(f_i)$ a sequence of smooth functions on $X$ such that $f_i \searrow f$.
\[\ddc\enE_{\mathrm{v}}^{\ddc f_i}(\Psi)=\int_X \mathrm{v}(m_{\Psi})\pi^{*}_X(\ddc f_i)\wedge(\pi^{*}_X\omega+\ddc\Psi)^{[n]}+\int_X\langle \mathrm{v}'(m_{\Psi}),m_{f_i}\rangle\Big(\pi^{*}_X\omega+\ddc\Psi\Big)^{[n+1]}.\]

As $i \to \infty$, we have $m_{f_i} \to m_f$ in the sense of currents, where $m_{f_i}^\xi = \dc f_i(\xi)$ and $m_f^\xi = \dc f(\xi)$. Both terms in the formula converge:
\begin{itemize}
\item The first term converges by weak convergence of the currents $\ddc f_i \to \ddc f$;
\item The second term converges since $\langle \mathrm{v}'(m_{\Psi}),m_{f_i}\rangle \to \langle \mathrm{v}'(m_{\Psi}),m_f\rangle$ weakly as currents.
\end{itemize}
Taking the limit as $i \to \infty$ yields:
\[\ddc\enE_{\mathrm{v}}^{\ddc f}(\Psi)=\int_X \mathrm{v}(m_{\Psi})\pi^{*}_X(\ddc f)\wedge(\pi^{*}_X\omega+\ddc\Psi)^{[n]}+\int_X\langle \mathrm{v}'(m_{\Psi}),m_f\rangle\Big(\pi^{*}_X\omega+\ddc\Psi\Big)^{[n+1]}.\]
Combining the smooth and singular parts completes the proof.
\end{proof}

 We consider the following family of elliptic boundary value problems with parameter $\varepsilon>0$,
\begin{equation}\label{MA-epsilon}
\begin{cases}
\big(\pi_{X}^*\omega+dd^{c}\Psi^{\varepsilon}\big)^{n+1}=\varepsilon\big(\pi_{X}^*\omega+\frac{\sqrt{-1}d\tau\wedge d\bar\tau}{2|\tau|^{2}}\big)^{n+1},\\
\Psi^{\varepsilon}(\cdot,e^{-1})=\f_1\text{ and }\Psi^{\varepsilon}(\cdot,1)=\f_0.
\end{cases}
\end{equation}
Solutions $\Psi^\varepsilon\in\cH^T(\hat{X},\pi^{*}_X\omega)^{\mathbb{G}}$ of \eqref{MA-epsilon} are always smooth and approximate uniformly the weak solution $\Psi$ of the boundary value problem \eqref{MA}. More precisely, $\Psi^{\varepsilon}$ is decreasing in $\varepsilon$ and converges to the solution $\Psi$ in the weak $C^{1,1}$ topology as $\varepsilon\to 0$. The family of $T$-invariant K\"ahler potentials $(\f^{\varepsilon}_t)_{t\in[0,1]}\subset\cH^T(X, \omega)$ is called an $\varepsilon$-\emph{geodesic}.

\begin{lemma}\label{lem-eq-ddc-func}
Let $(\varphi_t)_{t\in[0,1]} \subset \cH^{1,1}(X,\omega)^T$ be a weak geodesic segment with endpoints $\varphi_0, \varphi_1 \in \cH^T(X, \omega)$. Let $\Psi$ denote the corresponding $\mathbb{G}$-invariant solution of the boundary value problem \eqref{MA} on $\hat{X}$, and let $\chi$ be a positive $T$-invariant current. Then the following identities hold in the weak sense of currents:
\begin{align}
\ddc\enE_{\mathrm{v}}(\Psi)=&0,\label{eq-a}\\
\ddc\enE_{\mathrm{v}}^{\chi}(\Psi)=&\int_X \mathrm{v}(m_{\Psi})\pi^{*}_X\chi\wedge(\pi^{*}_X\omega+\ddc\Psi)^{[n]}.\label{eq-b}
\end{align}
\end{lemma}

\begin{proof}

    Let $(\f^{\varepsilon}_t)_{t\in [0,1]}$ be the $\varepsilon$-geodesic approximating $(\f_t)_{t\in [0,1]}$ and $\Psi^{\varepsilon}$ the corresponding solution of the elliptic Dirichlet problem. By Lemma \ref{lem-ddc-enE}, we have
\begin{align}
\begin{split}\label{eq-ddd}
\ddc\enE_{\mathrm{v}}(\Psi^{\varepsilon})=&\int_X \varepsilon \mathrm{v}(m_{\Psi^{\varepsilon}})\Big(\pi^{*}_X\omega+\frac{\sqrt{-1}d\tau\wedge d\bar\tau}{2|\tau|^{2}}\Big)^{[n+1]},\\
\ddc\enE_{\mathrm{v}}^{\chi}(\Psi^{\varepsilon})=&\int_X \mathrm{v}(m_{\Psi^\varepsilon})\pi^{*}_X\chi\wedge(\pi^{*}_X\omega+\ddc\Psi^\varepsilon)^{[n]}+\varepsilon\langle \mathrm{v}'(m_{\Psi^\varepsilon}),m_{\chi}\rangle\Big(\pi^{*}_X\omega+\frac{\sqrt{-1}d\tau\wedge d\bar\tau}{2|\tau|^{2}}\Big)^{[n+1]}.
\end{split}
\end{align}
We have $\Psi^{\varepsilon}\rightarrow \Psi$ in $(C^{1,1},\|\cdot\|_{C^1}+\|\ddc\cdot\|_{L^{\infty}})$ as $\varepsilon\to 0$. Using the identity
$$m_{\f^{\varepsilon}_t}=m_{\om}+m_{\ddc \f^{\varepsilon}_t}.$$ 
and the fact that $\mathrm{v}$ is smooth on $P$, we obtain
$$\enE_{\mathrm{v}}(\Psi^{\varepsilon})\rightarrow \enE_{\mathrm{v}}(\Psi)\text{ and }\enE_{\mathrm{v}}^{\chi}(\Psi^{\varepsilon})\rightarrow \enE_{\mathrm{v}}^{\chi}(\Psi),$$
since the Monge--Amp\`ere measures converge weakly under decreasing limits. It follows that 
$$\ddc\enE_{\mathrm{v}}(\Psi^{\varepsilon})\rightarrow \ddc\enE_{\mathrm{v}}(\Psi)\text{ and }\ddc\enE_{\mathrm{v}}^{\chi}(\Psi^{\varepsilon})\rightarrow \ddc\enE_{\mathrm{v}}^{\chi}(\Psi),$$
in the weak sense of distributions. 
Passing to the limit as $\varepsilon \to 0$ in \eqref{eq-ddd}:
\begin{itemize}
    \item The terms with the factor $\varepsilon$ vanish because $(\pi^{*}_X\omega+\ddc\Psi^\varepsilon)^{[n+1]}$ is uniformly bounded (due to the $C^{1,1}$ bound).
    \item The term $\pi^{*}_X\chi\wedge(\pi^{*}_X\omega+\ddc\Psi^\varepsilon)^{[n]}$ converges weakly to $\pi^{*}_X\chi\wedge(\pi^{*}_X\omega+\ddc\Psi)^{[n]}$ by the continuity of the mixed Monge--Amp\`ere operator for decreasing sequences of plurisubharmonic functions (note that $\Psi$ is $C^{1,1}$, so the wedge product is well-defined).
\end{itemize}
This yields \eqref{eq-a} and \eqref{eq-b}.

\end{proof}

\begin{lemma}\label{lem:divisorconvexsmooth}
Let $(\f_t)_{t\in[0,1]} \subset \cH^{1,1}(X,\omega)^T$ be a weak geodesic segment with endpoints $\varphi_0, \varphi_1 \in \cH^T(X, \omega)$ and $\Psi$ the corresponding $\mathbb{G}$-invariant solution of the boundary value problem \eqref{MA} on $\hat{X}$. Then the function $t \mapsto \enE_{\mathrm{v}}^{\chi}(\varphi_t)$ is convex.
\end{lemma}

\begin{proof}
By Lemma \ref{lem-eq-ddc-func}, we have:
\begin{align}
\ddc \enE_{\mathrm{v}}^{\chi}(\varphi_\tau) = \int_X \mathrm{v}(m_\Psi)\pi_X^*{\chi} \wedge (\pi_X^* \omega + dd^c \Psi)^{n}.
\end{align}
Since $\varphi_t$ is a subgeodesic, we have $\pi^* \omega + dd^c \Psi \geq 0$ on $\hat{X}$. Moreover, since $\mathrm{v} > 0$ on the moment polytope $P$ and ${\chi}$ is a positive $T$-invariant current, the integrand defines a positive current on $\mathbb{A}$. 

In the coordinates $\tau = e^{-t+is}$, this translates to:
\begin{align}
\frac{d^2}{dt^2} \enE_{\mathrm{v}}^{\chi}(\varphi_t) \geq 0,
\end{align}
establishing the convexity of $t \mapsto \enE_{\mathrm{v}}^{\chi}(\varphi_t)$.
\end{proof}

\begin{prop}\label{prop:extention-K-energy}
Let $(\varphi_t)_{t\in[0,1]} \subset \cH^{1,1}(X,\omega)^T$ be a weak geodesic segment with endpoints $\varphi_0, \varphi_1 \in \cH^T(X, \omega)$. Then the weighted twisted K-energy, decomposed as
\[
\cM_{\mathrm{v},\mathrm{w}}^{\chi}(\varphi_t) = \cM_{\mathrm{v},\mathrm{w}}(\varphi_t) + \enE_{\mathrm{v}}^{\chi}(\varphi_t),
\]
defines a convex function of $t$.
\end{prop}

\begin{proof}
By \cite[Theorem 6.1.]{AJL}, the weighted K-energy $\cM_{\mathrm{v},\mathrm{w}}(\f_t)$ is convex. By Lemma \ref{lem:divisorconvexsmooth}, $\enE_{\mathrm{v}}^{\chi}(\f_t)$ is convex. Since the sum of convex functions is convex, the result follows.
\end{proof}

\subsection{Convexity of weighted twisted K-energy}

As an application of the convexity of the twisted energy functional $\enE_{\mathrm{v}}^{\chi}$, we establish the convexity of the weighted twisted Mabuchi energy $\cM_{\mathrm{v},\mathrm{w}}^{\chi}$ along weak geodesics in the finite energy space $\cET(X, \omega)$.

\begin{thm}[see Theorem \ref{thm:main0}]\label{thm:main}
Let $(X,\omega)$ be a compact connected K\"ahler manifold and let $\chi$ be a positive current satisfying \eqref{eq: ChiProp}, and let $T\subset\Aut_{red}(X,\chi)$ be a compact torus. For smooth weights $\mathrm{v},\mathrm{w}\in C^\infty(\mathfrak{t}^\vee)$ with $\mathrm{v}>0$ on the moment polytope $P_X$, the weighted twisted Mabuchi functional
\begin{equation}
\cM_{\mathrm{v},\mathrm{w}}^{\chi}(\varphi):= 
\ent_{\mathrm{v}}(\varphi)
+  
\enR_{\mathrm{v}}^{\chi}(\varphi)+\enE_{\mathrm{v}\,\mathrm{w}}(\varphi)
\end{equation}
admits a unique greatest lower semicontinuous extension $\cM_{\mathrm{v},\mathrm{w}}^{\chi}: \cET(X, \omega) \to \R\cup \{+\infty\}$, that is moreover convex and continuous along weak geodesics.
\end{thm}
	\begin{proof}

	By Proposition \ref{prop:extention-K-energy}, $\cM_{\mathrm{v},\mathrm{w}}^{\chi}$ is convex on weak geodesics in $\cH^{1,1}(X,\omega)^T$ with smooth endpoints.

	To establish convexity on general weak geodesics in $\cET(X, \omega)$, let $t_0, t_1 \in [0,1]$ with $t_0 \leq t_1$ and let $(\f_t)_{t \in [t_0,t_1]} \subset \cET(X, \omega)$ be a $d_1$-geodesic. By Proposition \ref{prop:extension-mabuchi}, we can find sequences $\f^k_{t_0}, \f^k_{t_1} \in \cH^T(X, \omega)$ such that $d_1(\f^k_{t_0}, \f_{t_0}) \to 0$, $d_1(\f^k_{t_1}, \f_{t_1}) \to 0$, and
	$$\cM_{\mathrm{v},\mathrm{w}}^{\chi}(\f_{t_0}) = \lim_k \cM_{\mathrm{v},\mathrm{w}}^{\chi}(\f^k_{t_0}), \quad \cM_{\mathrm{v},\mathrm{w}}^{\chi}(\f_{t_1}) = \lim_k \cM_{\mathrm{v},\mathrm{w}}^{\chi}(\f^k_{t_1}).$$

	Let $[t_0,t_1] \ni t \mapsto \f^k_t \in \cH^{1,1}(X,\omega)^T$ be the weak geodesics connecting $\f^k_{t_0}$ and $\f^k_{t_1}$. By \cite[Prop.\ 4.3]{BDL}, we have $d_1(\f^k_t, \f_t) \to 0$ for any $t \in [t_0,t_1]$. 

	For any $t \in [t_0, t_1]$, the lower semicontinuity and convexity properties yield:
	\begin{align*}
	\cM_{\mathrm{v},\mathrm{w}}^{\chi}(\f_t) &\leq \liminf_k \cM_{\mathrm{v},\mathrm{w}}^{\chi}(\f^k_t) \\
	&\leq \liminf_k \left[ \frac{t_1-t}{t_1 - t_0} \cM_{\mathrm{v},\mathrm{w}}^{\chi}(\f^k_{t_0}) + \frac{t-t_0}{t_1 - t_0} \cM_{\mathrm{v},\mathrm{w}}^{\chi}(\f^k_{t_1}) \right]\\
	&= \frac{t_1-t}{t_1 - t_0} \cM_{\mathrm{v},\mathrm{w}}^{\chi}(\f_{t_0}) + \frac{t-t_0}{t_1 - t_0} \cM_{\mathrm{v},\mathrm{w}}^{\chi}(\f_{t_1}).
	\end{align*}
	The first inequality follows from the lower semicontinuity property of the extension. The second inequality uses the convexity of $\cM_{\mathrm{v},\mathrm{w}}^{\chi}$ along the geodesics $\f^k_t \in \cH^{1,1}(X,\omega)^T$. Hence, $t \mapsto \cM_{\mathrm{v},\mathrm{w}}^{\chi}(\f_t)$ is convex on $[t_0, t_1]$.
	The continuity follows from the lower semicontinuity and the convexity of $\cM_{\mathrm{v},\mathrm{w}}^{\chi}(\f_t)$.
	\end{proof}

\section{Openness of coercivity}
\subsection{Convexity of K-energy with mixed cusp and conic singularities}

Consider the case where $\chi$ is the current of integration along a simple normal crossing divisor $D=\sum_{i=1}^k D_i + \sum_{j=k+1}^N (1-\alpha_j)D_j$ where the first $k$ components $D_1,\ldots,D_k$ have coefficient $1$ (corresponding to cusp singularities with cone angle $0$), and the remaining components $D_{k+1},\ldots,D_N$ have coefficients $1-\alpha_j$ with $\alpha_j \in (0,1)$ (corresponding to conic singularities with cone angles $2\pi\alpha_j$). 

Taking the specialization of \eqref{eq: ChiProp} with
\begin{align}
f
&=
\sum_{i=1}^k
\left(
\log|s_{D_i}|_{h_i}^2
+
2\log(\lambda_i-\log|s_{D_i}|_{h_i}^2)
\right)
+
\sum_{j=k+1}^N
(1-\alpha_j)\log|s_{D_j}|_{h_j}^2,\\
\hat{f} &= -\sum_{i=1}^k \log(\lambda_i-\log|s_{D_i}|_{h_i}^2),\\
\chi &= 2\pi[D] = \sum_{i=1}^k 2\pi[D_i] + \sum_{j=k+1}^N 2\pi(1-\alpha_j)[D_j]
\end{align}
where $s_{D_i}$ are defining sections of the irreducible components $D_i$, and the Hermitian metrics $h_i$ are chosen so that $|s_{D_i}|_{h_i}<1$ on $X\setminus D_i$. The constants $\lambda_i$ are chosen sufficiently large so that $\omega-dd^c \sum_{i=1}^k 2\log(\lambda_i-\log|s_{D_i}|_{h_i}^2)>0$ on $X\setminus D_i$ (see Lemma \ref{lem:modelmetric}). This convention is compatible with
\[
2\pi[D_i]=\frac{1}{2}\ddc\log|s_{D_i}|_{h_i}^2+\theta_i.
\]
The term $\beta=\chi - \frac{1}{2}\ddc f-\ddc \hat{f}$ is a smooth $(1,1)$-form by the Lelong-Poincar\'e equation.
Let $\mathrm{v},\mathrm{w}\in C^\infty(\mathfrak{t}^\vee)$ be smooth weights with $\mathrm{v}>0$ on the moment polytope $P_X$.

\begin{remark}
The auxiliary potential $\hat{f}$ is crucial for handling Poincar\'e-type singularities. If one tries to use the framework of \cite{BDL} directly with only
\[
f=\sum_{i=1}^k \log|s_{D_i}|_{h_i}^2 \;+\; \sum_{j=k+1}^N (1-\alpha_j)\log |s_{D_j}|_{h_j}^2,
\]
then $2\pi[D]-\frac{1}{2}\ddc f$ is not a smooth form, so \cite{BDL} does not apply. Alternatively, one could consider $\cM_{\mathrm{v},\mathrm{w}}^{2\pi[D]-\ddc \hat{f}}$, which does fit into the framework of \cite{BDL}, but the resulting Euler--Lagrange equation no longer takes the desired form
\[
(\mathrm{w} - S_{\mathrm{v}}^{2\pi[D]}(\omega_\varphi))\,\MA_{\mathrm{v}}(\varphi)=0.
\]
\end{remark}

\begin{cor}[Weighted K-energy convexity for mixed cusp and conic singularities]\label{prop:extention-K-energy-divisor}
Let $(X,\omega)$ be a compact connected K\"ahler manifold and let $T\subset\Aut_{red}(X,[D])$ be a compact torus. For smooth weights $\mathrm{v},\mathrm{w}\in C^\infty(\mathfrak{t}^\vee)$ with $\mathrm{v}>0$ on the moment polytope $P_X$, the weighted twisted Mabuchi functional
\begin{equation}
\cM_{\mathrm{v},\mathrm{w}}^{2\pi[D]}(\varphi) = \ent_{\mathrm{v}}(\varphi) + \enR_{\mathrm{v}}^{2\pi[D]}(\varphi) + \enE_{\mathrm{v}\,\mathrm{w}}(\varphi)
\end{equation}
admits a unique greatest lower semicontinuous extension $\cM_{\mathrm{v},\mathrm{w}}^{2\pi[D]}: \cET(X, \omega) \to \R\cup \{+\infty\}$, that is moreover convex and continuous along weak geodesics.
\end{cor}

\begin{proof}
	By Lemma \ref{lem:poincare-finite-volume}, we have $ e^{-f}\in L^1(X,\omega)$, and by Lemma \ref{lem:poincare-E1}, $ \hat f \in \cET(X, \omega)$. The result follows directly from our general framework in Theorem \ref{thm:main}.
\end{proof}
\subsection{Coercivity}

In this section, we establish the openness of the coercivity condition for the (relative) weighted twisted Mabuchi energy with respect to the cone angle parameters. Our main result, Theorem \ref{thm:coeropen0}, demonstrates that the coercivity of the relative functional $\cM^\mathrm{rel}_{\mathrm{v},\mathrm{w},2\pi [D]}$ is stable under small perturbations of the angle vector $\boldsymbol{\alpha} = (\alpha_1, \dots, \alpha_N)$ associated with the simple normal crossing divisor $D = \sum (1-\alpha_j)D_j$. We prove that if coercivity holds for a given configuration of cone angles, it persists for all parameters in a neighborhood. The proof relies on comparing weighted entropies via Lemma~\ref{lemma:Ent-compare} and applying the finite-energy estimates \eqref{equ:wenlip}, \eqref{equ:twenhold}, and \eqref{equ:energy-ddc-finite-continuity}. We first establish the necessary preliminary results regarding the relative entropy in the variational framework of \cite{BBEGZ}.

Let $X$ be an arbitrary compact (Hausdorff) topological space, fixed with a positive Radon measure $\nu$. Recall that the \emph{relative entropy} of any positive Radon measure $\mu$ on $X$ with respect to $\nu$ is defined as
$$
\Ent(\mu|\nu)=\int_X f\log f\,d\nu\in\R\cup\{+\infty\}
$$
if $d\mu=f d\nu$ with $f\in L^1(\nu)$, and $\Ent(\mu|\nu)=+\infty$ otherwise. A crucial feature of the relative entropy is its characterization as a Legendre transform. As established in \cite[Lemma 2.11]{BBEGZ}, we have:
\begin{equation}\label{equ:entleg}
\Ent(\mu|\nu)=\sup_{g\in C^0(X)}\left\{\int_X g\,d\mu - \mu(X)\log\int_X e^g d\nu\right\} + \mu(X)\log\mu(X).
\end{equation}
This formula implies that the functional
$$
\Ent(\cdot|\nu)\colon\cM\to\R\cup\{+\infty\}
$$
is convex and lower semicontinuous (lsc) on the space $\cM$ of positive Radon measures equipped with the weak topology. Furthermore, it satisfies the bound:
\begin{equation}\label{equ:entbd}
\Ent(\mu|\nu)\ge\mu(X)\log\frac{\mu(X)}{\nu(X)}.
\end{equation}
We note that the variational formula \eqref{equ:entleg} remains valid when the test function $g$ is merely upper semicontinuous (usc) and bounded above (or lsc and bounded below), via monotone approximation by continuous functions (see \cite[Lemma 2.11]{BBEGZ}).

\begin{lemma}[Entropy Stability]\label{lemma:Ent-compare}
Let $X$ be a compact topological space equipped with a probability measure $\nu$. Let $F: X \to \mathbb{R} \cup \{+\infty\}$ be a lower or upper semicontinuous function bounded from below. Define the positive Radon measure $\hat{\nu} = e^{-F} \nu$.

Suppose there exist constants $p > 1$ such that $\int_X e^{-pF} d\nu$ is finite,
then for any probability measure $\mu$ on $X$,
\[
\Ent(\mu|\nu) \le \left(\frac{p}{p-1}\right) \Ent(\mu|\hat{\nu}) + \frac{\eta}{p-1}.
\]
where $\eta = \log \int_X e^{-pF} d\nu$.
\end{lemma}

\begin{proof}
The density of $\mu$ relative to $\hat{\nu}$ is given by $\frac{d\mu}{d\hat{\nu}} = \frac{d\mu}{d\nu} e^F$. Expanding the relative entropy yields the identity:
\begin{equation}\label{eq:bridge}
\Ent(\mu|\hat{\nu}) = \int_X \log\left(\frac{d\mu}{d\nu} e^F\right) d\mu = \Ent(\mu|\nu) + \int_X F \, d\mu.
\end{equation}
To bound the term $-\int_X F \, d\mu$, we utilize the Legendre transform characterization \eqref{equ:entleg}. If $F$ is lower semicontinuous, then the function $g = -pF$ is upper semicontinuous. We have:
\[
\int_X g \, d\mu \le \log \int_X e^g d\nu + \Ent(\mu|\nu).
\]
Substituting $g = -pF$, we obtain:
\[
-p \int_X F \, d\mu \le \log \int_X e^{-pF} d\nu + \Ent(\mu|\nu) \le \log C + \Ent(\mu|\nu).
\]
Dividing by $p$ gives:
\[
-\int_X F \, d\mu \le \frac{1}{p} \Ent(\mu|\nu) + \frac{1}{p} \log C.
\]
Substituting this bound into the rearranged identity $\Ent(\mu|\nu) = \Ent(\mu|\hat{\nu}) - \int_X F d\mu$, we get:
\[
\Ent(\mu|\nu) \le \Ent(\mu|\hat{\nu}) + \frac{1}{p} \Ent(\mu|\nu) + \frac{1}{p} \log C.
\]
Collecting the $\Ent(\mu|\nu)$ terms on the left side:
\[
\left(1 - \frac{1}{p}\right) \Ent(\mu|\nu) \le \Ent(\mu|\hat{\nu}) + \frac{1}{p} \log C.
\]
Multiplying by $\frac{p}{p-1}$ yields the result.
\end{proof}

Recall the notation $\ent_{\mathrm{v}}^{f}(\varphi) := \frac{1}{2}\Ent(\MA_{\mathrm{v}}(\varphi) \mid e^{-f} \omega^n)$ for the weighted entropy with respect to the measure $e^{-f}\omega^n$. Following the notation of Subsection~\ref{subsec:futaki}, we abbreviate
$\ell^{\mathrm{ext}}:=\ell^{\mathrm{ext}}_{\omega,\mathrm{v},\mathrm{w},\chi}$ and $\hat{\ell}^{\mathrm{ext}}:=\ell^{\mathrm{ext}}_{\omega,\mathrm{v},\mathrm{w},\hat{\chi}}$.

\begin{lemma}[Affineness of the extremal function in the twist]\label{lem:ellext-affine}
Fix $\omega$ and smooth weights $\mathrm{v},\mathrm{w}\in C^\infty(\ft^\vee)$ with $\mathrm{v},\mathrm{w}>0$ on $P$.
Let $\chi$ and $\hat{\chi}$ be $T$-invariant $(1,1)$-currents of order $0$.
Set $\cA:=\ft\oplus\R$ (affine functions on $\ft^\vee$), and define
\[
I(\ell'):=\langle\,\cdot\,,\ell'\rangle,\qquad 
L_\chi(\ell):=\int_X \ell(m_\omega)\,S_{\mathrm{v}}^\chi(\omega)\,\MA_{\mathrm{v}}(0),\qquad \ell,\ell'\in\cA.
\]
Then:
$\ell^{\mathrm{ext}}_{\omega,\mathrm{v},\mathrm{w},\chi}=I^{-1}(L_\chi)$, and for any $\hat{\chi}$
\begin{equation}\label{eq:ellext-affine}
\ell^{\mathrm{ext}}_{\hat{\chi}}-\ell^{\mathrm{ext}}_{\chi}
=I^{-1}(L_{\hat{\chi}}-L_{\chi})
=-\,I^{-1}\!\left(\ell\mapsto \int_X \ell(m_\omega)\,\MA_{\mathrm{v}}^{\hat{\chi}-\chi}(0)\right).
\end{equation}
In particular, $\ell^{\mathrm{ext}}_{\omega,\mathrm{v},\mathrm{w},\chi}$ depends affinely on $\chi$.

\end{lemma}

\begin{proof}
The statement follows by subtracting the characterization \eqref{equ:ellext-twisted} for the two twists and using the pairing \eqref{equ:Futaki-pairing}.
\end{proof}

\begin{prop}[Openness of coercivity under perturbations]\label{prop:coercivity-weighted-twisted}
Fix smooth weights $\mathrm{v},\mathrm{w}\in C^\infty(\ft^\vee)$ with $\mathrm{v},\mathrm{w}>0$ on the (normalized) moment polytope $P$.
Let $\chi$ be a positive $(1,1)$-current
\[
\chi=\beta+\tfrac12\ddc f,
\]
where $\beta$ is smooth. Fix a smooth real $(1,1)$-form $\mathcal{B}$ and a quasi-psh function $F$, set
\[
\chi_t:=\beta_t+\tfrac12\ddc f_t,\qquad 
\beta_t:=\beta+t\mathcal{B},\qquad f_t:=f+tF,
\]
 and choose $t_0 >0$ sufficiently small such that $\chi_t$ is positive for all $t\in[0,t_0]$.
Assume:
\begin{enumerate}
\item[(i)] There exists $p>1$ such that $e^{-F}\in L^p(X,e^{-f}\omega^n)$.
\item[(ii)] The weighted twisted Mabuchi functional
\[
\cM^\mathrm{rel}_{\mathrm{v},\mathrm{w},\chi}(\varphi)
=\ent_{\mathrm{v}}^{f}(\varphi)+\enR_{\mathrm{v}}^{\beta}(\varphi)+\enE_{\mathrm{v}\,\mathrm{w}\,\ell_\chi^{\mathrm{ext}}}(\varphi)
\]
is coercive relative to $T^\mathbb{C}$, i.e.\ there exist $\delta,A>0$ such that
\[
\cM^\mathrm{rel}_{\mathrm{v},\mathrm{w},\chi}(\varphi)\ge \delta\, d_{1,T^\mathbb{C}}(\varphi,0)-A,
\]
for all $\varphi\in\cET_0(X,\omega):=\{\f \in \cET | \enE(\f)=0\}$.
\end{enumerate}
Then there exists a constant $C>0$, depending only on the background data
$(\omega,\mathrm{v},\mathrm{w},\chi, \mathcal{B})$, such that for every $t\in(0,t_0]$, and every
$\varphi\in\cET_0(X,\omega)$,
\[
\cM^\mathrm{rel}_{\mathrm{v},\mathrm{w},\chi_t}(\varphi)\ge \delta_t\, d_{1,T^\mathbb{C}}(\varphi,0)-A_t,
\]
where
\[
\eta:=\int_X e^{-pF}\,e^{-f}\,\omega^n,\qquad 
\delta_t:=\delta\Bigl(1-\frac{t}{p}\Bigr)-Ct,\qquad 
A_t:=A+\frac{t}{p-t}\log\eta.
\]
In particular, as \(t\to 0\) we have \(\delta_t\to \delta\) and $A_t\to A$, and the functional \(\cM^\mathrm{rel}_{\mathrm{v},\mathrm{w},\chi_t}\) remains coercive relative to \(T^\mathbb{C}\).

\end{prop}

\begin{proof}
Throughout the proof, $C$ denotes a positive constant that may change from line to line.
For each $t$, the functional $\cM^\mathrm{rel}_{\mathrm{v},\mathrm{w},\chi_t}$ is $T^\mathbb{C}$-invariant.

Set $\nu:=e^{-f}\omega^n$ and $\nu_t:=e^{-f_t}\omega^n$. We have
\[
\nu_t=e^{-tF}\nu.
\]
Assumption (i) says that $e^{-pF}\in L^1(X,\nu)$.
For $t\in(0,t_0]$, we have $\frac{p}{t}>1$ and $\frac{p}{t}\to+\infty$ as $t\to0^+$. Also, $e^{-(\frac{p}{t})tF}=e^{-pF}$ is $\nu$-integrable, so Lemma~\ref{lemma:Ent-compare}
applied to the pair $(\nu,\nu_t)$ with exponent $\frac{p}{t}$ yields
\[
\ent_{\mathrm{v}}^{f}(\varphi)\le \frac{p}{p-t}\,\ent_{\mathrm{v}}^{f_t}(\varphi)+\frac{t}{p-t}\log \eta,
\]
hence, after rearranging,
\[
\ent_{\mathrm{v}}^{f_t}(\varphi)\ge \Bigl(1-\frac{t}{p}\Bigr)\ent_{\mathrm{v}}^{f}(\varphi)-\frac{t}{p-t}\log \eta
=\Bigl(1-\frac{t}{p}\Bigr)\ent_{\mathrm{v}}^{f}(\varphi)-\frac{t}{p-t}\log \eta.
\]

We now expand $\cM^\mathrm{rel}_{\mathrm{v},\mathrm{w},\chi_t}(\varphi)$ and insert the previous inequality:
\begin{align*}
\cM^\mathrm{rel}_{\mathrm{v},\mathrm{w},\chi_t}(\varphi)
&=\ent_{\mathrm{v}}^{f_t}(\varphi)+\enR_{\mathrm{v}}^{\beta_t}(\varphi)+\enE_{\mathrm{v}\,\mathrm{w}\,\ell_{\chi_t}^{\mathrm{ext}}}(\varphi)\\
&\ge \Bigl(1-\frac{t}{p}\Bigr)\ent_{\mathrm{v}}^{f}(\varphi)+\enR_{\mathrm{v}}^{\beta}(\varphi)+\enE_{\mathrm{v}\,\mathrm{w}\,\ell_{\chi}^{\mathrm{ext}}}(\varphi)
+\enE_{\mathrm{v}}^{t\mathcal{B}}(\varphi)
+\enE_{\mathrm{v}\,\mathrm{w}\,(\ell_{\chi_t}^{\mathrm{ext}}-\ell_\chi^{\mathrm{ext}})}(\varphi)
-\frac{t}{p-t}\log \eta\\
&=\Bigl(1-\frac{t}{p}\Bigr)\cM^\mathrm{rel}_{\mathrm{v},\mathrm{w},\chi}(\varphi)
+\frac{t}{p}\Bigl(\enR_{\mathrm{v}}^{\beta}(\varphi)+\enE_{\mathrm{v}\,\mathrm{w}\,\ell_\chi^{\mathrm{ext}}}(\varphi)\Bigr)
+\enE_{\mathrm{v}}^{t\mathcal{B}}(\varphi)
+\enE_{\mathrm{v}\,\mathrm{w}\,(\ell_{\chi_t}^{\mathrm{ext}}-\ell_\chi^{\mathrm{ext}})}(\varphi)
-\frac{t}{p-t}\log \eta.
\end{align*}

By the finite-energy estimates \eqref{equ:wenlip}, \eqref{equ:twenhold}, and \eqref{equ:energy-ddc-finite-continuity}, we have
\[
\Bigl|\enR_{\mathrm{v}}^{\beta}(\varphi)+\enE_{\mathrm{v}\,\mathrm{w}\,\ell_\chi^{\mathrm{ext}}}(\varphi)\Bigr|\le C\,d_1(\varphi,0),
\qquad
\bigl|\enE_{\mathrm{v}}^{t\mathcal{B}}(\varphi)\bigr|\le C\,\|t\mathcal{B}\|_\omega\, d_1(\varphi,0),
\]
and, using linearity in the weight,
\[
\bigl|\enE_{\mathrm{v}\,\mathrm{w}\,(\ell_{\chi_t}^{\mathrm{ext}}-\ell_\chi^{\mathrm{ext}})}(\varphi)\bigr|
\le C\,\|\mathrm{w}\|_{C^0(P)}\,\|\ell_{\chi_t}^{\mathrm{ext}}-\ell_{\chi}^{\mathrm{ext}}\|_{C^0(P)}\,d_1(\varphi,0).
\]
We get $\|t\mathcal{B}\|_\omega\le Ct$.
Moreover, by Lemma~\ref{lem:ellext-affine},
\[
\|\ell_{\chi_t}^{\mathrm{ext}}-\ell_{\chi}^{\mathrm{ext}}\|_{C^0(P)}\le Ct.
\]
Therefore the error terms satisfy
\[
\frac{t}{p}\Bigl|\enR_{\mathrm{v}}^{\beta}(\varphi)+\enE_{\mathrm{v}\,\mathrm{w}\,\ell_\chi^{\mathrm{ext}}}(\varphi)\Bigr|
+\bigl|\enE_{\mathrm{v}}^{t\mathcal{B}}(\varphi)\bigr|
+\bigl|\enE_{\mathrm{v}\,\mathrm{w}\,(\ell_{\chi_t}^{\mathrm{ext}}-\ell_\chi^{\mathrm{ext}})}(\varphi)\bigr|
\le C\Bigl(\frac{t}{p}+t\Bigr)d_1(\varphi,0)\le Ct\,d_1(\varphi,0),
\]
after absorbing constants into $C$.

By Remark~\ref{rem:relative_dist_attained} and $T^\mathbb{C}$-invariance of $\cM^\mathrm{rel}_{\mathrm{v},\mathrm{w},\chi_t}$, we can assume that $d_{1,T^\mathbb{C}}(\varphi,0)=d_1(\varphi,0)$.
We obtain
\[
\cM^\mathrm{rel}_{\mathrm{v},\mathrm{w},\chi_t}(\varphi)
\ge \delta_t\,d_{1,T^\mathbb{C}}(\varphi,0)-A-\frac{t}{p-t}\log \eta,
\]
with
\(
\delta_t=\delta\left(1-\frac{t}{p}\right)-Ct.
\)
\end{proof}

\begin{prop}\label{prop:cusp-to-conic}
Fix smooth weights $\mathrm{v},\mathrm{w}\in C^\infty(\ft^\vee)$ with $\mathrm{v},\mathrm{w}>0$ on the (normalized) moment polytope $P$.
Suppose that
\[
\cM^\mathrm{rel}_{\mathrm{v},\mathrm{w},\chi}(\varphi) \geq \delta\, d_{1,T^\mathbb{C}}(\varphi, 0) - A
\]
for all $\varphi \in \cET_0(X,\omega)$.
Let $\psi$ be a function with $0\leq \psi<1$ and let $\beta$ be a smooth $(1,1)$-form. Set
\[
\chi_t:=\chi-t\,(\ddc\log\psi+\,\beta),
\]
 and choose $t_0 >0$ sufficiently small such that $\chi_t$ is positive for all $t\in[0,t_0]$.
Then there exists a constant $C>0$ depending only on the background information $(\omega,\mathrm{v},\mathrm{w},\chi, \ddc\log\psi+\beta)$ such that
\[
\cM^\mathrm{rel}_{\mathrm{v},\mathrm{w},\chi_t}(\varphi) \geq (\delta-Ct)\, d_{1,T^\mathbb{C}}(\varphi, 0) - A
\]
for all $\varphi \in \cET_0(X,\omega)$.
\end{prop}

\begin{proof}
In this proof, $C$ denotes a generic positive constant independent of $t$ and $\varphi$, which may change from line to line.
Write $\ell_\chi^{\mathrm{ext}}:=\ell^{\mathrm{ext}}_{\omega,\mathrm{v},\mathrm{w},\chi}$ and
$\ell_{\chi_t}^{\mathrm{ext}}:=\ell^{\mathrm{ext}}_{\omega,\mathrm{v},\mathrm{w},\chi_t}$.
By Lemma~\ref{lem:ellext-affine},
\[
\|\mathrm{w}(\ell_{\chi_t}^{\mathrm{ext}}-\ell_\chi^{\mathrm{ext}})\|_{C^0(P)}\le Ct.
\]

For any $\varphi \in \cH^T(X, \omega)$, the decomposition of the energy functionals gives
\[
\cM^\mathrm{rel}_{\mathrm{v},\mathrm{w},\chi_t}(\varphi) - \cM_{\mathrm{v},\mathrm{w} \ell_{\chi_t}^{\mathrm{ext}}}^{\chi}(\varphi)
= -t \int_X \log \psi \,\MA_{\mathrm{v}}(\varphi) - t \,\enE_{\mathrm{v}}^{\beta}(\varphi).
\]
Since $0<\psi<1$, we have $-\log\psi>0$, so the first term is non-negative. Passing to the greatest lsc extensions on
$\cET(X,\omega)$ yields, for all $\varphi\in\cET(X,\omega)$,
\[
\cM^\mathrm{rel}_{\mathrm{v},\mathrm{w},\chi_t}(\varphi) - \cM_{\mathrm{v},\mathrm{w} \ell_{\chi_t}^{\mathrm{ext}}}^{\chi}(\varphi)
\ge -t\,\enE_{\mathrm{v}}^{\beta}(\varphi).
\]
Using linearity in the weight,
\[
\cM_{\mathrm{v},\mathrm{w} \ell_{\chi_t}^{\mathrm{ext}}}^{\chi}(\varphi)
=\cM^\mathrm{rel}_{\mathrm{v},\mathrm{w},\chi}(\varphi)
+\enE_{\mathrm{v}\,\mathrm{w}\,(\ell_{\chi_t}^{\mathrm{ext}}-\ell_\chi^{\mathrm{ext}})}(\varphi).
\]
Applying the finite-energy estimates \eqref{equ:wenlip} and \eqref{equ:twenhold} to $\enE_{\mathrm{v}}^\beta$ and to $\enE_{\mathrm{v}\,\mathrm{w}\,(\ell_{\chi_t}^{\mathrm{ext}}-\ell_\chi^{\mathrm{ext}})}$ gives
\[
t\,\bigl|\enE_{\mathrm{v}}^\beta(\varphi)\bigr|\le Ct\, d_1(\varphi,0),
\qquad
\bigl|\enE_{\mathrm{v}\,\mathrm{w}\,(\ell_{\chi_t}^{\mathrm{ext}}-\ell_\chi^{\mathrm{ext}})}(\varphi)\bigr|\le Ct\, d_1(\varphi,0).
\]

and therefore
\[
\cM^\mathrm{rel}_{\mathrm{v},\mathrm{w},\chi_t}(\varphi)
\ge \cM^\mathrm{rel}_{\mathrm{v},\mathrm{w},\chi}(\varphi) - Ct\, d_1(\varphi,0).
\]

By Remark~\ref{rem:relative_dist_attained} and $T^\mathbb{C}$-invariance of $\cM^\mathrm{rel}_{\mathrm{v},\mathrm{w},\chi_t}$, we can assume that $d_{1,T^\mathbb{C}}(\varphi,0)=d_1(\varphi,0)$.
Hence, using the previous estimate and the coercivity hypothesis for $\cM^\mathrm{rel}_{\mathrm{v},\mathrm{w},\chi}$,
\begin{align*}
\cM^\mathrm{rel}_{\mathrm{v},\mathrm{w},\chi_t}(\varphi)
&\ge \cM^\mathrm{rel}_{\mathrm{v},\mathrm{w},\chi}(\varphi) - Ct\, d_1(\varphi,0)\\
&\ge \delta\, d_1(\varphi,0)-A - Ct\,d_{1,T^\mathbb{C}}(\varphi,0)\\
&\ge (\delta-Ct)\, d_{1,T^\mathbb{C}}(\varphi,0) - A,
\end{align*}

\end{proof}

\begin{remark}[Perturbations in several directions]\label{rem:multi-direction-perturb}
Notice $\delta_t$ depends linearly on $t$ in Propositions~\ref{prop:coercivity-weighted-twisted} and~\ref{prop:cusp-to-conic}. It is easy to see that by induction on the number of directions, the result extends to multi-parameter perturbations. Indeed, given smooth forms $B_i$ and quasi-psh functions $F_i$ ($1\le i\le N$), one sets
\[
\beta_{\mathbf t}:=\beta+\sum_{i=1}^N t_i \mathcal{B}_i,\qquad
f_{\mathbf t}:=f+\sum_{i=1}^N t_i F_i,\qquad
\chi_{\mathbf t}:=\beta_{\mathbf t}+\tfrac12\ddc f_{\mathbf t},
\]
and assumes $e^{-F_i}\in L^{p_i}(X,e^{-f}\omega^n)$ for some $p_i>1$. An analogous result holds for
$\cM^\mathrm{rel}_{\mathrm{v},\mathrm{w},\chi_{\mathbf t}}(\varphi)$.
\end{remark}

\subsubsection*{Applications to conic and cusp singularities}

Having established the general openness result for coercivity under perturbations (Proposition \ref{prop:coercivity-weighted-twisted} and Proposition \ref{prop:cusp-to-conic}), we now apply it to two important geometric settings: perturbations from cusp singularities to conic singularities, and perturbations of cone angles for conic singularities.

The following theorem shows that coercivity is an open condition with respect to the cone angle parameters.

\begin{thm}[Openness of coercivity in cone angles]\label{thm:coeropen0}
Fix smooth weights $\mathrm{v},\mathrm{w}\in C^\infty(\ft^\vee)$ with $\mathrm{v},\mathrm{w}>0$ on the (normalized) moment polytope $P$.
Let \(D = \sum_{j=1}^N (1-\alpha_j)D_j\) be a simple normal crossing divisor and consider the twist class $\chi = 2\pi [D]$, where $\boldsymbol{\alpha}=(\alpha_1,\dots,\alpha_N)\in [0,1]^{N}$. Suppose there exist $\delta>0$ and $A \in \R$ such that

$$\cM^\mathrm{rel}_{\mathrm{v},\mathrm{w},\chi}(\varphi) \geq \delta \, d_{1, T^\mathbb{C}}(\varphi,0) - A,$$

for all $\varphi \in \cET_0(X, \omega)$.
Then, for any $\hat{\delta} < \delta$, there exist $\hat{A} \in \R$ and a neighborhood $\mathcal{U}$ of $\boldsymbol{\alpha}$ in $[0,1]^{N}$ such that

$$\cM^\mathrm{rel}_{\mathrm{v},\mathrm{w},\hat{\chi}}(\varphi) \geq \hat{\delta} \, d_{1, T^\mathbb{C}}(\varphi,0) - \hat{A}$$

for all $\varphi \in \cET_0(X,\omega)$ and any $\hat{\boldsymbol{\alpha}} \in \mathcal{U}$, where $\hat{\chi}=\sum_{j=1}^N 2\pi(1-\hat{\alpha}_j)[D_j]$. Moreover, $\delta-\hat \delta$ is linearly controlled by $\boldsymbol{\alpha} - \hat{\boldsymbol{\alpha}}$.

\end{thm}

\begin{proof}


By Remark~\ref{rem:multi-direction-perturb}, we may assume without loss of generality that
\[
D=(1-\alpha_1)D_1,
\] i.e. $D$ has only one component, and we are perturbing the single angle $\alpha_1$.

For simplicity, we henceforth view $D$ as a smooth hypersurface of $X$ and set
\[
\chi:=2\pi(1-\alpha)[D],\qquad \hat\chi:=2\pi(1-\hat\alpha)[D].
\]
We write their Ricci decompositions as
\[
\chi=(1-\alpha)\theta+\tfrac12\ddc f,\qquad \hat\chi=(1-\hat\alpha)\theta+\tfrac12\ddc \hat f,
\]
where, for a defining section $s_D$ of $\mathcal O(D)$ endowed with a smooth Hermitian metric $h$,
\[
f=(1-\alpha)\log|s_D|_h^{2},\qquad
\hat f=(1-\hat\alpha)\log|s_D|_h^{2},
\]
with $\theta$ the curvature form of $h$. In particular,
\[
f-\hat f=(\hat\alpha-\alpha)\log|s_D|_h^{2}.
\]

\paragraph{Case 1.}
Assume $1\ge \hat\alpha\ge \alpha$. Choosing $h$ so that $|s_D|_h^{2}\le 1$ on $X$, we have $\log|s_D|_h^{2}\le 0$.
Applying Proposition~\ref{prop:cusp-to-conic} then yields
\[
\cM^\mathrm{rel}_{\mathrm{v},\mathrm{w},\hat\chi}(\varphi)
\ge \bigl(\delta-C(\hat\alpha-\alpha)\bigr)\, d_{1,T^\C}(\varphi,0)-A,
\]
for all $\varphi\in\cET_0(X,\omega)$.
Hence, for any fixed $\hat\delta<\delta$, if $\hat\alpha\ge\alpha$ is sufficiently close to $\alpha$ so that
$C(\hat\alpha-\alpha)\le \delta-\hat\delta$, we get
\[
\cM^\mathrm{rel}_{\mathrm{v},\mathrm{w},\hat\chi}(\varphi)\ge \hat\delta\, d_{1,T^\C}(\varphi,0)-A.
\]

\paragraph{Case 2.}
Assume $\alpha/2\le \hat\alpha<\alpha$. Define the reference current at angle $\alpha/2$ by
\[
\chi_{1/2}:=2\pi\Bigl(1-\frac{\alpha}{2}\Bigr)[D]
=\Bigl(1-\frac{\alpha}{2}\Bigr)\theta+\tfrac12\ddc \Bigl(1-\frac{\alpha}{2}\Bigr)\log|s_D|_h^{2},
\]

Set
\[
F:=\frac{\alpha}{2}\log|s_D|_h^{2},\qquad 
\mathcal B:=\frac{\alpha}{2}\theta,
\]
and
\[
t:=\frac{2(\alpha-\hat\alpha)}{\alpha}\in(0,1],\qquad
\eta:=\int_X e^{-pF}e^{-f}\,\omega^n.
\]
Then
\[
\hat f=f+tF,\qquad 
(1-\hat\alpha)\theta=(1-\alpha)\theta+t\mathcal B,
\]
and consequently
\[
\hat\chi=\chi+t\Bigl(\mathcal B+\tfrac12\ddc F\Bigr)
=\bigl((1-\alpha)\theta+t\mathcal B\bigr)+\tfrac12\ddc\bigl(f+tF\bigr).
\]

By Lemma~\ref{lem:integrability-SNC}, $e^{-F}\in L^p(X,e^{-f}\omega^n)$ for every $1<p<\frac{2}{\alpha}$.
Fix any $p\in(1,\frac{2}{\alpha})$.
Applying Proposition~\ref{prop:coercivity-weighted-twisted} (with this $t$ and $p$) gives, for all
$\varphi\in\cET_0(X,\omega)$,
\[
\cM^\mathrm{rel}_{\mathrm{v},\mathrm{w},\hat \chi}(\varphi)
\ge \delta_{t}\, d_{1,T^\C}(\varphi,0)-A-\frac{t}{p-t}\log \eta,
\]
Set
\[
\hat A:=A+\frac{t}{p-t}\log\eta,\qquad 
\hat\delta:=\delta\Bigl(1-\frac{t}{p}\Bigr)-Ct.
\]
Then \(\hat\alpha\to\alpha^{-}\) implies \(t\to 0\) (and hence \(t/p\to 0\)), so that
\[
\hat\delta\to\delta
\qquad\text{and}\qquad
\hat A\to A.
\]
\end{proof}

\begin{remark}
    Although we assumed $\delta > 0$, the proofs in this section remain valid for any $\delta \in \mathbb{R}$, as the derived estimates are independent of the sign of the coercivity constant.
\end{remark}

\section{Appendix}
\subsection{The Log-Mabuchi energy of Poincar\'e type K\"ahler metrics}\label{subsec:log-mabuchi-poincare}
\subsubsection{Potential space of Poincar\'e type K\"ahler metrics}\label{sec:poincare-type}

We briefly recall the construction of the model metric $\omega_{\rm mod}$; let $(X, \omega)$ be a compact K\"ahler manifold of complex dimension $n$, in which we consider a divisor $D$ with \textit{simple normal crossings} with decomposition $D=\sum_{j=1}^N D_j$ into smooth irreducible components. 
For each $j=1,\dots,N$, let $s_{D_j}$ be a holomorphic defining section of $\mathcal{O}(D_j)$, and let $h_j$ be a Hermitian metric on $\mathcal{O}(D_j)$. We can assume that
$$
\rho_j:=-\log(|s_{D_j}|_{h_j}^2)\geq 1
$$
outside $D_j$.
Now let $\lambda$ be a nonnegative real parameter. If we set $u_j:=\log(\lambda+\rho_j)=\log\big(\lambda-\log(|s_{D_j}|_{h_j}^2)\big)$, one has:
 \begin{lemma}[\cite{Auvray}, Lemma 1.1]\label{lem:modelmetric}
  Let $A>0$. For sufficiently big $\lambda$ (depending on $A$ and $\omega$), the (1,1)-form $\omega-A\frac{1}{2}\ddc u_j$ defines a K\"ahler form on $X\backslash D_j$.
 \end{lemma}
 This follows from a direct computation: 
 \begin{equation*}
-A \frac{1}{2} dd^c u_j = \frac{ A\frac{1}{2} d\rho_j \wedge d^c\rho_j }{(\lambda+\rho_j)^2} - \frac{A \frac{1}{2} dd^c \rho_j}{\lambda+\rho_j}
 \end{equation*}
The first summand is a nonnegative (1,1)-form, whereas $\pm\tfrac{A\frac{1}{2}\ddc\rho_j}{\lambda+\rho_j}\leq \tfrac{CA}{\lambda+\rho_j}\omega$ in the sense of (1,1)-forms where $C$ is such that $\pm \frac{1}{2}\ddc\rho_j \leq C\omega$ on $X$. Since $\rho_j \to +\infty$ near $D_j$, we have $\omega-\tfrac{A\frac{1}{2}\ddc\rho_j}{\lambda+\rho_j}>0$ on $X\backslash D_j$ when $\lambda$ is large enough. $\square$

~

By choosing $A_1,\dots, A_N>0$, replacing $\omega$ with $\tfrac{1}{N}\omega$, and increasing $\lambda$ if necessary, we obtain $\tfrac{1}{N}\omega-\frac{1}{2}\ddc u_j>0$ on $X\backslash D_j$ for $j=1,\dots,N$, hence
$$
\omega_{\rm mod} = \omega - \frac{1}{2}\ddc\mathfrak{u} = \sum_{j=1}^N \left(\frac{1}{N}\omega - A_j \frac{1}{2}\ddc u_j\right).
$$ 
(recall $\mathfrak{u}=\sum_{j=1}^N A_j u_j$) defines a genuine K\"ahler form on $X\setminus D$. 
 
\begin{defi}[Poincar\'e type K\"ahler metric {\cite[Def. 0.1]{Auvray}, \cite[Def. 1.1]{Auvray1}}]
\label{def: poincaretypemetdefn}
Let $X$ be a compact complex manifold and let $D$ be a smooth irreducible divisor in $X$. Let $\Omega \in H^2(X, \mathbb{R})$ be a K\"ahler class. A smooth, closed, real $(1,1)$ form $\omega_\varphi$ on $X \setminus D$ is a \emph{Poincar\'e type K\"ahler metric} if

	\begin{itemize}

		\item $\omega_\varphi$ is quasi-isometric to $\om_{\rm mod}$. That is, there exists a $C$ such that
		\begin{align*} C^{-1} \om_{\rm mod} \leq \omega_\varphi \leq C \om_{\rm mod}.
		\end{align*}

		Moreover, the class of $\omega_\varphi$ is $\Omega$ if

		\item $\omega_\varphi = \omega + \frac{1}{2}\ddc \varphi$ for a smooth function $\varphi$ on $X \setminus D$ with $| \nabla_{\omega_{\rm mod}}^j \varphi |$ bounded for any $j \in \mathbb{N}^*$ and $\varphi = O(\mathfrak{u})$.

	\end{itemize}
    Such $\f$ is called a \emph{Poincar\'e type K\"ahler potential}. The space of Poincar\'e type K\"ahler potentials is denoted by $\tmom$.
\end{defi}

\noindent By Lemma \ref{lem:poincare-E1}, the space $\tmom$ of Poincar\'e type potentials is contained in the finite energy space $\mathcal{E}^1$. We denote $\tmom^T:= \tmom \cap \cET(X, \omega)$ the space of $T$-invariant Poincar\'e type K\"ahler potentials.

Without loss of generality, we can assume that the K\"ahler metric $\om$ and the potential $\mathfrak{u}$ are $T$-invariant.

\begin{lemma}\label{lem:poincare-finite-volume}
The volume of any Poincar\'e type metric $\om_\f$ is finite. More precisely, the volume is the topological constant $[\om]^n$.
\end{lemma}

\begin{proof}
The finiteness of volume follows from an elementary calculation in local coordinates. Without loss of generality, we may assume that $D$ is smooth and is defined by $\{z_1=0\}$ in a local chart $(z_1,\ldots,z_n)$. By the asymptotic behavior of Poincar\'e type metrics (Definition~\ref{def: poincaretypemetdefn}), the volume form satisfies
\begin{equation}\label{eq:MA-asymptotic}
C^{-1}\frac{1}{|z_1|^2\log^2(|z_1|)}\omega^n \leq \omega_\varphi^n \leq C\frac{1}{|z_1|^2\log^2(|z_1|)}\omega^n
\end{equation}
for some constant $C>0$. Thus it suffices to show that
\begin{equation}\label{eq:poincare-volume-integral}
\int_{\mathbb{D}_\varepsilon} \frac{1}{|z|^2\log^2(|z|)} i dz \wedge d\bar{z} < +\infty
\end{equation}
for sufficiently small $\varepsilon > 0$, where $\mathbb{D}_\varepsilon = \{z \in \mathbb{C} : |z| < \varepsilon\}$. Using polar coordinates $z=re^{i\theta}$, we have
\begin{align*}
\int_{\mathbb{D}_\varepsilon} \frac{1}{|z|^2\log^2(|z|)} i dz \wedge d\bar{z} &= \int_0^{2\pi} \int_0^{\varepsilon} \frac{1}{r^2\log^2(r)} 2r \, dr \, d\theta \\
&= 2\pi \int_0^{\varepsilon} \frac{1}{r\log^2(r)} dr \\
&= 2\pi \int_{-\infty}^{\log(\varepsilon)} \frac{1}{u^2} du \quad (u=\log(r)) \\
&= 2\pi \left[ -\frac{1}{u} \right]_{\log(\varepsilon)}^{-\infty} = \frac{2\pi}{|\log(\varepsilon)|} < +\infty.
\end{align*}
This establishes the finiteness of volume near the divisor, and thus the total volume is finite.
\end{proof}

For the weighted Monge--Amp\`ere operator of Poincar\'e type metrics, the asymptotic estimate~\eqref{eq:MA-asymptotic} established extends to the general case with weight $\mathrm{v}$: there exists a constant $C>0$ such that
\begin{equation}\label{eq:MA-v-asymptotic}
C^{-1}\frac{1}{|s_D|_h^2 (\log|s_D|_h)^2}\omega^n \leq \MA_{\mathrm{v}}(\f) \leq C\frac{1}{|s_D|_h^2 (\log|s_D|_h)^2}\omega^n
\end{equation}
near the divisor. By Corollary \ref{decomp}, the Ricci curvature satisfies
\begin{equation}\label{eq:Ricci-asymptotic}
-C \omega_{\rm mod} \leq \Ric(\omega_\varphi) - 2\pi[D] \leq C\omega_{\rm mod}.
\end{equation}
Combining these two-sided estimates, we obtain the bounded scalar curvature property
\begin{equation}\label{eq:scalar-curvature-bounded}
S_{\mathrm{v}}^{2\pi[D]}(\omega_{\f}):= \frac{\MA_{\mathrm{v}}^{\Ric_{\mathrm{v}}(\om_\f)-2\pi[D]}(\f)}{\MA_{\mathrm{v}}(\f)}=O(1).
\end{equation}

We now specialize the general framework to Poincar\'e type metrics by fixing $\chi=2\pi[D]$, where $D=\sum_{j=1}^N D_j$ is a simple normal crossing divisor. For each $j=1,\dots,N$, let $s_{D_j}$ be a holomorphic defining section of $\mathcal{O}(D_j)$, and let $h_j$ be a Hermitian metric on $\mathcal{O}(D_j)$ with curvature form $\theta_j$.
By the Lelong--Poincar\'e formula,
\[
2\pi [D_j] = \frac{1}{2}\ddc \log |s_{D_j}|_{h_j}^2 + \theta_j,
\]
we write
\[
\chi = \beta + \frac{1}{2}\ddc f,
\]
with
\[
f = \log|s_D|_h^2,\qquad
\beta = \sum_{j=1}^N \theta_j.
\]
For $\f \in C^\infty(X)^T$, the $2\pi[D]$-twisted Mabuchi functional decomposes as
\[
\cM_{\mathrm{v},\mathrm{w}}^{2\pi[D]}(\f)=\ent_{\mathrm{v}}^{\log|s_D|_h^2}(\f)+\enR_{\mathrm{v}}^{\ddc \log h}(\f)+\enE_{\mathrm{v}\,\mathrm{w}}(\f).
\]
We will show below that this formula remains well-defined for $\f \in \tmom^T$.

\begin{lemma}\label{lem:integrability-SNC}

Let $X$ be a compact K\"ahler manifold and $D = \sum_{j=1}^N D_j$ be a simple normal crossing divisor. Let $u = \sum_{j=1}^N \gamma_j \log |s_{D_j}|_{h_j}^2$ for some real constants $\gamma_j$. Then $e^{-u} \in L^1(X, \omega^n)$ if and only if $\gamma_j < 1$ for all $j=1, \dots, N$.

\end{lemma}

\begin{proof}

The integrability is a local property. Since $D$ has simple normal crossing support, for any point $x \in X$, there exists a local coordinate chart $(U, z_1, \dots, z_n)$ centered at $x$ such that $D$ is given by $\{z_1 \cdots z_k = 0\}$ for some $k \le n$. In this chart, the hermitian norm satisfies $|s_{D_j}|_{h_j}^2 \approx |z_j|^2$ (up to a smooth non-vanishing factor). The volume form satisfies $\omega^n \approx \prod_{m=1}^n \frac{i}{2} dz_m \wedge d\bar{z}_m$.

Thus, locally, the convergence of the integral is equivalent to the convergence of:

$$\int_{\mathbb{D}^n} \prod_{j=1}^k |z_j|^{-2\gamma_j} \prod_{m=1}^n \frac{i}{2} dz_m \wedge d\bar{z}_m.$$

By the Fubini-Tonelli theorem, this integral splits into a product of integrals over each coordinate disk $\mathbb{D}$. For the non-singular directions $j > k$, the integrals are trivial. For $j \le k$, using polar coordinates $z_j = r_j e^{i\theta_j}$, we have $\frac{i}{2} dz_j \wedge d\bar{z}_j = r_j dr_j d\theta_j$. The component integral becomes:

$$\int_0^1 r_j^{-2\gamma_j} r_j dr_j = \int_0^1 r_j^{1-2\gamma_j} dr_j.$$

This integral converges if and only if the exponent satisfies $1-2\gamma_j > -1$, which simplifies to $2\gamma_j < 2$, or $\gamma_j < 1$. Since this must hold for every chart covering $D$, the result follows.

\end{proof}

\begin{lemma}\label{lem:poincare-E1}
	The space $\tmom$ of Poincar\'e type potentials is contained in the finite energy space $\mathcal{E}^1(X, \omega)$.
	\end{lemma}

	\begin{proof}
By definition, we need to show
		\[
			\int_X |\varphi| \, \omega_\f^n
			\]
			is integrable.
			 
			Near $D=\{z_1 = 0\}$, in polar coordinates $z_1 = re^{i\theta}$, this is controlled by
			\[
			2\pi \int_0^{\varepsilon}  \frac{\log(-\log r) }{r\log^2(r)} dr.
			\]
			Let $u = -\log r$, so $du = -\frac{dr}{r}$. As $r \to 0^+$, we have $u \to +\infty$, and when $r = \varepsilon$, we have $u = -\log \varepsilon$. Thus
			\[
			2\pi \int_0^{\varepsilon}  \frac{\log(-\log r) }{r\log^2(r)} dr = 2\pi \int_{-\log\varepsilon}^{+\infty} \frac{\log u}{u^2} du < +\infty,
			\]
			for $\varepsilon$ sufficiently small. 
		\end{proof}

\begin{remark}
	For simplicity and consistency with the literature, although we assume that all higher derivatives of the potential are bounded by the model Poincar\'e type metric, for the remainder of the discussion we only require the boundedness of the first four derivatives.
\end{remark}

		\begin{lemma}\label{lem:entropy-finite-poincare}
			The entropy functional 
			\[
			\ent_{\mathrm{v}}^{\log|s_D|_h^2}(\varphi) = \frac{1}{2}\Ent\left(\MA_{\mathrm{v}}(\varphi) \,\Big|\; e^{-\log|s_D|_h^2}\om^n\right)
			\]
			is finite for $\varphi \in \tmom^T$.
		\end{lemma}

\begin{proof}
		By definition, the entropy is given by
		\[
		\ent_{\mathrm{v}}^{f}(\varphi) = \frac{1}{2} \int_X \log\left(\frac{\MA_{\mathrm{v}}(\varphi)}{e^{-f}\omega^n}\right) \MA_{\mathrm{v}}(\varphi).
		\]
		Using the asymptotic estimate \eqref{eq:MA-v-asymptotic}, $\MA_{\mathrm{v}}(\varphi)$ is mutually bounded with $\frac{1}{|s_D|_h^2 (\log |s_D|_h)^2} \omega^n$ (in the sense that their ratio is bounded from above and below). 
		Since $e^{-f} = \frac{1}{|s_D|_h^2}$, we have the mutual bound
		\[
		C^{-1} \frac{1}{(\log |s_D|_h)^2} \leq \frac{\MA_{\mathrm{v}}(\varphi)}{e^{-f}\omega^n} \leq C \frac{1}{(\log |s_D|_h)^2}.
		\]

		Therefore, $\log\left(\frac{\MA_{\mathrm{v}}(\varphi)}{e^{-f}\omega^n}\right)$ is mutually bounded with $-2\log(-\log|s_D|_h)$. The convergence of the integral is then equivalent to the integrability of
		\[
		\int_X \frac{-\log(-\log|s_D|_h)}{|s_D|_h^2 (\log|s_D|_h)^2} \omega^n,
		\]
        which is finite by the same calculation as in Lemma \ref{lem:poincare-E1}.
		\end{proof}

	\begin{lemma}\label{lem:ricci-energy-finite-poincare}
    The weighted twisted Ricci energy $\enR_{\mathrm{v}}^{\ddc \log h}(\varphi)$ is finite for any Poincar\'e type metric $\om_\varphi$.
\end{lemma}
\begin{proof}
	    The relevant twist form $\eta:=-\Ric(\omega)+\ddc\log h$ is smooth on $X$. Hence the finite-energy estimate \eqref{equ:twenhold}, applied with $\p=0$, gives
	    \[
	    |\enE_{\mathrm{v}}^\eta(\varphi)-\enE_{\mathrm{v}}^\eta(0)|
	    \leq C\,d_1(\varphi,0).
	    \]
	    The result follows from Lemma~\ref{lem:poincare-E1}.
\end{proof}

	\begin{cor}\label{cor:mabuchi-welldef-poincare}
	The weighted twisted Mabuchi energy $\cM_{\mathrm{v},\mathrm{w}}^{2\pi[D]}$ is finite on $\tmom^T$.
	\end{cor}

\begin{lemma}\label{lem:mixed_holder_bound}
Let $\varphi \in \cET(X, \omega)$. If $g \in \cET(X, \omega)$ and $f \in C^\infty(X)$, then
    \begin{equation}\label{equ:twMA_holder1}
    \left| \int_X g \MA_{\mathrm{v}}^{\ddc f}(\varphi) \right| \leq C \dd_1(0,g)^\alpha \max{\big(\dd_1(0,f),\dd_1(0,\varphi)\big)}^{1-\alpha},
    \end{equation}
    where $\alpha=2^{-n}$ and $C > 0$ depends on $\sup_P |\mathrm{v}|$, $\sup_P |\mathrm{v}'|$, and $|\ddc f |_{\omega}$.
\end{lemma}

\begin{proof}
    By Equation \eqref{equ:twMA}, we have:
    \begin{equation}\label{eq:ibp_proof}
    \int_X g \MA_{\mathrm{v}}^{\ddc f}(\varphi) = n \int_X \mathrm{v}(m_{\omega_\varphi}) g (\omega-\omega_f) \wedge \omega_\varphi^{n-1} + \int_X g \langle \mathrm{v}'(m_{\omega_\varphi}), m_{f}\rangle \omega_\varphi^n.
    \end{equation}
     For the first term, applying the mixed energy estimates from \cite{BBGZ,DDL}, we obtain:
    \[
    \left| \int_X \mathrm{v}(m_{\omega_\varphi}) g (\omega-\omega_f) \wedge \omega_\varphi^{n-1} \right| \leq (\sup_P |\mathrm{v}|) \, C \, \dd_1(0,g)^\alpha \max{\big(\dd_1(0,f),\dd_1(0,\varphi)\big)}^{1-\alpha}.
    \]
     Similarly the second term we have,
	 \[
	 \left| \int_X g \langle \mathrm{v}'(m_{\omega_\varphi}), m_{f}\rangle \omega_\varphi^n \right| \leq \, C \, \dd_1(0,g)^\alpha \dd_1(0,\varphi)^{1-\alpha}.
	 \]    
	 where $C$ depends on $\sup_P |\mathrm{v}'|$ and $|\ddc f|_\omega$. 
    Combining these estimates yields the result.
\end{proof}

\begin{lemma}\label{lem:symmetry_E1_general}
Let $\varphi \in \cET(X, \omega)$. If $f$ and $g$ are both in $\cET(X, \omega)$ and at least one is smooth, then
    \begin{equation}\label{equ:twMA_symmetry_final}
 \int_X g \MA_{\mathrm{v}}^{\ddc f}(\varphi)  = \int_X f \MA_{\mathrm{v}}^{\ddc g}(\varphi).
    \end{equation}
\end{lemma}

\begin{proof}
   Assume $f$ is smooth. When $\varphi$ is smooth, the result follows from the smooth symmetry identity \eqref{equ:twMAsymm}. 

    Otherwise, we take a sequence of smooth functions $\varphi_k$ decreasing to $\varphi$. We apply the mixed energy estimate \eqref{equ:twMA_holder1} from Lemma~\ref{lem:mixed_holder_bound} to the functional $\psi \mapsto \int g \MA_{\mathrm{v}}^{\ddc f}(\psi)$. Since $\varphi_k \to \varphi$ in the $d_1$-topology, the estimate implies:
    \[
    \lim_{k \to \infty} \int_X g \MA_{\mathrm{v}}^{\ddc f}(\varphi_k) = \int_X g \MA_{\mathrm{v}}^{\ddc f}(\varphi).
    \]
	By the definition, $\int_X f \MA_{\mathrm{v}}^{\ddc g}(\varphi)$ is the limit of $\int_X f \MA_{\mathrm{v}}^{\ddc g}(\varphi_k)$ as $k \to \infty$ which concludes the proof.
\end{proof}

As a consequence, we obtain:

\begin{lemma}
	For any $\f \in \cET(X, \om)$, the differential of $\enE_{\mathrm{v}}$ is given by $\MA_{\mathrm{v}}(\f)$, and for any $T$-invariant $(1,1)$-form $\chi$, the differential of $\enE_{\mathrm{v}}^{\chi}$ is given by $\MA_{\mathrm{v}}^{\chi}(\f)$. That is,
	\[
	\langle (\enE_{\mathrm{v}})'(\f), \psi \rangle = \int_X \psi \cdot \MA_{\mathrm{v}}(\f), \quad \langle (\enE_{\mathrm{v}}^{\chi})'(\f), \psi \rangle = \int_X \psi \cdot \MA_{\mathrm{v}}^{\chi}(\f)
	\]
	for any $\psi \in  \cET(X, \om)$. 

	In particular, for $\chi$ smooth and $\f, \psi \in \tmom$, these variations are well-defined and finite.
\end{lemma}

\begin{proof}
 The variation formulas follow from the smooth identities \eqref{equ:Ev-first-var} and \eqref{equ:Evchi-first-var}, together with Lemma~\ref{lem:symmetry_E1_general}. The finiteness for $\tmom$ is established by the computations in the proof of Lemma~\ref{lem:poincare-finite-volume}.
\end{proof}

We turn to consider the variation of the entropy term $\ent_{\mathrm{v}}^{f}(\f)$ on the space of $T$-invariant Poincar\'e type potentials $\tmom^T$.

\begin{lemma}\label{lem:variation-entropy-poincare}
    The variation of the entropy term $\ent_{\mathrm{v}}^{\log|s_D|_h^2}(\f)$ on the space of $T$-invariant Poincar\'e type potentials $\tmom^T$ is well-defined and finite, and satisfies
    \[
    \langle (\ent_{\mathrm{v}}^{\log|s_D|_h^2})'(\f), \psi \rangle = \int_X \psi \left(-\MA_{\mathrm{v}}^{\Ric_{\mathrm{v}}(\om_\f)}(\f) + \MA_{\mathrm{v}}^{\Ric(\om)}(\f) + \MA_{\mathrm{v}}^{\frac{1}{2}\ddc(\log|s_D|_h^2)}(\f)\right)
    \]
    for any $\psi \in \tmom^T$. The variation is finite.
\end{lemma}

\begin{proof}
    Let $f = \log|s_D|_h^2$. We consider the expansion of $\ent_{\mathrm{v}}^{f}(\f+t\psi)$ around $t=0$. Using the expansion
    \[
    \MA_{\mathrm{v}}(\f+t\psi) = \MA_{\mathrm{v}}(\f) + t \MA_{\mathrm{v}}^{\ddc \psi}(\f) + O(t^2) = \MA_{\mathrm{v}}(\f) \left( 1 + t \frac{\MA_{\mathrm{v}}^{\ddc \psi}(\f)}{\MA_{\mathrm{v}}(\f)} + O(t^2) \right),
    \]
    understood in the weak sense of measures, we have
    \[
    \log \left(\frac{\MA_{\mathrm{v}}(\f+t\psi)}{e^{-f}\om^n}\right) = \log \left(\frac{\MA_{\mathrm{v}}(\f)}{e^{-f}\om^n}\right) + t \frac{\MA_{\mathrm{v}}^{\ddc \psi}(\f)}{\MA_{\mathrm{v}}(\f)} + O(t^2).
    \]
    Substituting this into the entropy functional yields:
    \begin{equation*}
    \begin{split}
        \frac{d}{dt}\Big|_{t=0} \ent_{\mathrm{v}}^{f}(\f+t\psi) &= \frac{1}{2} \int_X \left( \frac{\MA_{\mathrm{v}}^{\ddc \psi}(\f)}{\MA_{\mathrm{v}}(\f)} \right) \MA_{\mathrm{v}}(\f) + \frac{1}{2} \int_X \log \left(\frac{\MA_{\mathrm{v}}(\f)}{e^{-f}\om^n}\right) \MA_{\mathrm{v}}^{\ddc \psi}(\f) \\
        &= \frac{1}{2} \int_X \MA_{\mathrm{v}}^{\ddc \psi}(\f) + \frac{1}{2} \int_X \psi \MA_{\mathrm{v}}^{\ddc \log \left(\frac{\MA_{\mathrm{v}}(\f)}{e^{-f}\om^n}\right) }(\f),
    \end{split}
    \end{equation*}
    In the second equality, the first term vanishes by the mass conservation of the weighted Laplacian. For the second term, we applied Lemma \ref{lem:symmetry_E1_general}, which is justified since $\log \left(\frac{\MA_{\mathrm{v}}(\f)}{e^{-f}\om^n}\right)$ is smooth and $\psi \in \cET$ (see Lemma \ref{lem:poincare-E1}).
    
    Using the definitions of weighted Ricci curvature $\Ric_{\mathrm{v}}(\om_\f) = -\frac{1}{2}\ddc \log (\mathrm{v}(m_{\om_\f})\om_\f^n)$ and $\Ric(\om) = -\frac{1}{2}\ddc \log \om^n$, we compute:
    \begin{equation*}
    \begin{split}
        \frac{1}{2}\ddc \log \left(\frac{\MA_{\mathrm{v}}(\f)}{e^{-f}\om^n}\right) &= \frac{1}{2}\ddc \log (\mathrm{v}(m_{\om_\f})\om_\f^n) - \frac{1}{2}\ddc \log \om^n + \frac{1}{2}\ddc f \\
        &= -\Ric_{\mathrm{v}}(\om_\f) + \Ric(\om) + \frac{1}{2}\ddc \log|s_D|_h^2.
    \end{split}
    \end{equation*}
    Thus, the variation formula becomes
    \[
    \langle (\ent_{\mathrm{v}}^{f})'(\f), \psi \rangle = \int_X \psi \left(-\MA_{\mathrm{v}}^{\Ric_{\mathrm{v}}(\om_\f)}(\f) + \MA_{\mathrm{v}}^{\Ric(\om)}(\f) + \MA_{\mathrm{v}}^{\frac{1}{2}\ddc \log|s_D|_h^2}(\f)\right).
    \]
    
    It remains to show that the variation is finite. Using the decomposition from Corollary \ref{decomp}, we have
    \[
    \Ric(\omega_\varphi) = \theta_\varphi + \theta'_\varphi + 2\pi[D],
    \]
    where $\theta_\varphi$ is smooth and $\theta'_\varphi$ is controlled by $\om_{\rm mod}$. 
    Combined with the Lelong-Poincar\'e equation, it follows that the density of the variation
    \[
    -\Ric_{\mathrm{v}}(\omega_\varphi) + \Ric(\omega) + \ddc \log|s_D|_h
    \]
    is controlled by $\om_{\rm mod}$. Since $\psi$ is of Poincar\'e-type, by the calculation in the proof of Lemma \ref{lem:poincare-E1}, the integral is finite.
\end{proof}

\begin{remark}
By Proposition~\ref{prop:extension-mabuchi}, the approximation procedure used in this section yields the same greatest lower semicontinuous extension of the relevant weighted twisted Mabuchi functional.
\end{remark}

As a corollary, we obtain the first variation formula for Poincar\'e type potentials.

\begin{prop}[First variation for Poincar\'e type potentials]\label{prop:first-var-weak-poincare}
Let $\f \in \tmom^T$ be a Poincar\'e type potential and $\psi \in  \tmom^T$. Then the first variation formula
\[
\langle (\cM_{\mathrm{v},\mathrm{w}}^{2\pi[D]})'(\f),\psi \rangle = \int_X \psi \cdot (\mathrm{w} - S_{\mathrm{v}}^{2\pi[D]}(\omega_\f)) \, \MA_{\mathrm{v}}(\f)
\]
holds.
\end{prop}

\begin{proof}	
   This is a direct consequence of the following decomposition,
    \[
    \cM_{\mathrm{v},\mathrm{w}}^{2\pi[D]}(\f) = \ent_{\mathrm{v}}^{\log|s_D|_h^2}(\f) + \enR_{\mathrm{v}}^{\ddc\log h}(\f) + \enE_{\mathrm{v}\,\mathrm{w}}(\f).
    \]
	    This decomposition defines the same functional as the one defined using the greatest lower semicontinuous extension in Proposition~\ref{prop:extension-mabuchi}. Specifically, we can write
    \[
    \cM_{\mathrm{v},\mathrm{w}}^{2\pi[D]}(\f) = \ent_{\mathrm{v}}^{f}(\f) + \enR_{\mathrm{v}}^{\ddc \hat{f}}(\f) + \enE_{\mathrm{v}\,\mathrm{w}}(\f),
    \]
    where
	\begin{align}
		f &= \sum_{j=1}^N  \left(\log|s_{D_j}|_{h_j}^2 + 2\log(\lambda_j-\log|s_{D_j}|_{h_j}^2)\right),\\
	\hat{f} &= -\sum_{j=1}^N  \log(\lambda_j-\log|s_{D_j}|_{h_j}^2).
	\end{align}
    This equality holds because the difference term $\sum_{j=1}^N \int_X \log(\lambda_j-\log|s_{D_j}|_{h_j}^2) \MA_{\mathrm{v}}(\f)$ is absolutely integrable by Lemma \ref{lem:poincare-E1}.

\end{proof}

\begin{remark}
    This formula agrees, up to a constant, with the one defined in \cite{Auvray}.
\end{remark}

\subsection{Curvature decomposition of Poincar\'e type metrics}\label{sec:poincare-curvature}

In this appendix, we establish the curvature decomposition formula for Poincar\'e type metrics that is used throughout the paper. The main result (Corollary~\ref{decomp}) shows that the Ricci curvature of a Poincar\'e type metric decomposes into a smooth part, a logarithmic part, and the divisor $[D]$. The key technical tool is a variant of the Poincar\'e-Lelong formula (Proposition~\ref{PoincareLelong}) adapted to functions with specific asymptotic behavior near the divisor, which we establish using a variant of Cauchy's integral formula.

	\begin{lemma}[A Variant of Cauchy's Formula]\label{cauchy}
		Let \( f \) be a smooth function defined on a neighborhood of the closed disc \( \overline{\mathbb{D}}_\varepsilon \subset \mathbb{C}\), and let \( \partial \overline{\mathbb{D}}_\varepsilon \) denote the positively oriented boundary. Suppose \( g \) is a smooth and bounded function defined outside the origin satisfying
		\[
		\lim_{z \to 0} \left|z \frac{\partial g}{g\partial z}\right| = 0.
		\]
		Then the Cauchy integral formula is given by:
		\[
		0 = \frac{1}{2\pi i} \lim_{\varepsilon \to 0} \int_{\partial \mathbb{D}_\varepsilon} f \, d \log|g|^2.
		\]
	\end{lemma}
	
	\begin{proof}
		First, observe that

		\[
		\int_{\partial \mathbb{D}_\varepsilon} f \, d \log|g|^2 = \int_{\partial \mathbb{D}_\varepsilon} f \left(\frac{z\partial |g|}{|g|\partial z} \frac{dz}{z} + \frac{\bar z\partial |g|}{|g|\partial \bar{z}}\frac{d\bar{z}}{\bar z}\right).
		\]
		Since $|z| = \varepsilon$ on $\partial \mathbb{D}_\varepsilon$ and using the condition that $\lim_{z \to 0} \left|z \frac{\partial g}{g\partial z}\right| = 0$, we have
		\[
		\left|\int_{\partial \mathbb{D}_\varepsilon} f \, d \log|g|^2\right| \leq C \cdot 2\pi\varepsilon \to 0
		\]
		as $\varepsilon \to 0$, where $C = \sup_{\partial \mathbb{D}_\varepsilon} |f|$.
		This completes the proof.
	\end{proof}

The following \emph{Variant of Poincar\'e-Lelong Formula} will provide a current-based description of the Ricci curvature of a Poincar\'e type metric. This formula decomposes the Ricci curvature into a smooth part and a divisorial part. As we have not found a proof of this formula in the literature, we establish it using the above \emph{Variant of Cauchy's Integral Formula}.

\begin{prop}[A Variant of the Poincar\'e-Lelong Formula]\label{PoincareLelong}
Let \((X, \omega)\) be a K\"ahler manifold. Let \(g\) be a smooth function on \(X \setminus D\). For any holomorphic chart \(U\) near the divisor \(D\) where \(D_{\text{smooth}} \cap U = \{z_1 = 0\}\), assume the following conditions hold:
\[
\lim_{z_1 \to 0} \left|z_1 \frac{\partial g}{g\partial z_1}\right| = 0.
\]
Let \(N_\varepsilon\) denote an \(\varepsilon\)-neighborhood of \(D\) with respect to the metric \(\omega\). Then, the following holds:
\[
\frac{1}{2\pi i} \lim_{\varepsilon \to 0} \int_{\partial N_\varepsilon} d \log |g|^2 \wedge \Lambda = 0,
\]
where \(\Lambda\) is a (possibly continuous) \((n-1, n-1)\)-form.

Equivalently, by Stokes' theorem, we have:
\[
\frac{1}{2\pi i} \lim_{\varepsilon \to 0} \int_{N_\varepsilon} \ddc \log |g|^2 \wedge \Lambda = 0,
\]
and in particular, \(\ddc \log |g|^2\) is \(L^1_{\text{loc}}\).
\end{prop}

\begin{proof}
First, since \( D_{\mathrm{singular}} \) has complex codimension at least two in \( X \), by \cite[Chapter III, Corollary 2.11]{Demailly}, the current \( \ddc \log |g|^2 \) defines the same current when restricted to \( X \setminus D_{\mathrm{singular}} \).

Observe that
\[
d \log |g|^2 \wedge \Lambda = \left( \frac{\partial |g|}{|g| \partial z_1} dz_1 + \frac{\partial |g|}{|g| \partial \bar{z}_1} d\bar{z}_1 \right) \wedge \Lambda.
\]
Consider an open subset of $X$ parameterized by \(\mathbb{D}_\varepsilon \times N\), where \(N\) is an open subset of \(D\).

By Fubini's theorem and \(\Lambda\) is a continuous form hence \(|\Lambda|_{\omega} = O(1)\), we have
\[
\int_{\partial \mathbb{D}_\varepsilon \times N} \left( \frac{\partial |g|}{|g| \partial z_1} dz_1 + \frac{\partial |g|}{|g| \partial \bar{z}_1} d\bar{z}_1 \right) \wedge \Lambda = O\left( \int_{\partial \mathbb{D}_\varepsilon} \left( \frac{\partial |g|}{|g| \partial z_1} dz_1 + \frac{\partial |g|}{|g| \partial \bar{z}_1} d\bar{z}_1 \right) \right).
\]
By Lemma~\ref{cauchy}, we have
\[
\lim_{\varepsilon \to 0} \int_{\partial \mathbb{D}_\varepsilon \times N} \left( \frac{\partial |g|}{|g| \partial z_1} dz_1 + \frac{\partial |g|}{|g| \partial \bar{z}_1} d\bar{z}_1 \right) \wedge \Lambda = 0.
\]
Since \(N_\varepsilon\) can be covered by finitely many such neighborhoods, the result follows.
\end{proof}

		As a corollary of this theorem, we establish the following decomposition formula for the Ricci curvature of a Poincar\'e type metric. Furthermore, in subsequent discussions, $\theta_\varphi$ will consistently denote the part of the Ricci curvature that is smooth on $X\setminus D$ and controlled near $D$, as defined in the following corollary.
	
		\begin{cor}[Decomposition of the Ricci curvature]\label{decomp}
			The Ricci curvature of a Poincar\'e-type metric as a current decomposes (not uniquely) into three components:
		\[
		\Ric(\omega_\varphi) = \theta_\varphi +\theta'_\f+ 2\pi[D] ,
		\]
		where:
		\begin{itemize}
			\item \(\theta_\varphi=\ddc \rho\) where $\rho$ is a bounded function smooth on $X\setminus D$,
				\item \(\theta'_\f=\ddc \log|\log|s_D|_h||^2\) is an $L^1_{\rm loc}$-form on $X$, smooth on $X\setminus D$, and satisfies \(-C\om_{\rm mod}\leq\theta'_\f\leq C \om_{\rm mod}\) on $X\setminus D$ for some constant $C>0$,
			\item \(2\pi[D]\) is the current of integration along the divisor \(D\), representing the singular contribution of the Ricci curvature concentrated on \(D\).
		\end{itemize}

\end{cor}
	
\begin{proof}
We establish the decomposition by analyzing the Ricci curvature in local coordinates near the divisor. By definition of Poincar\'e type metrics, there exists a globally bounded function $\rho$ that is smooth on $X\setminus D$ such that

\[
\omega_\varphi^n = \rho \frac{1}{|s_D|_h^2 \log^2|s_D|_h^2} \omega^n.
\]

Taking $dd^c\log$ of both sides formally gives 
\[\Ric(\omega_\varphi) = \ddc\log\rho - \ddc\log\left(|s_D|_h^2\log^2|s_D|_h^2\right) + \text{bounded terms}.\]

\medskip
\noindent\textbf{Local coordinates analysis.} For any holomorphic chart $U$ near the smooth part of the divisor $D$, where $D_{\text{smooth}} \cap U = \{z_1 = 0\}$, the volume form of $\omega_\varphi$ can be locally expressed as:
\[
\omega_\varphi^n = \psi \frac{1}{|z_1|^2 \log^2|z_1|^2} \bigwedge_{j=1}^n (i \, dz_j \wedge d\bar{z}_j),
\]
where $\psi$ is a bounded strictly positive function. This function $\psi$ satisfies the asymptotic condition
\begin{equation}\label{eq:psi-asymptotic}
\lim_{z_1 \to 0} z_1 \frac{\partial |\psi|}{|\psi|\partial z_1} = 0 \quad \text{and} \quad \lim_{z_1 \to 0} z_1 \frac{\partial |\psi|}{|\psi|\partial \bar{z}_1} = 0.
\end{equation}
Consequently, by Proposition~\ref{PoincareLelong}, $dd^c \log|\psi|$ is an $L^1_{\text{loc}}$-form.

Since the fourth-order derivatives of $\varphi$ are bounded under the model metric $\om_{\rm mod}$ according to Definition \ref{def: poincaretypemetdefn}, the partial derivatives of the metric coefficients $g_{j\bar{k}}$ of $\omega_\varphi$ exhibit the following asymptotic behavior:\[
\frac{\partial g_{j\bar{k}}}{\partial z_1} = O\left(\frac{1}{|z_1| \log|z_1|}\right), \quad \frac{\partial g_{j\bar{k}}}{\partial \bar{z}_1} = O\left(\frac{1}{|z_1| \log|z_1|}\right),
\]
while for $i \geq 2$:
\[
\frac{\partial g_{j\bar{k}}}{\partial z_i} = O(1), \quad \frac{\partial g_{j\bar{k}}}{\partial \bar{z}_i} = O(1).
\]
Although $\frac{\partial g_{j\bar{k}}}{\partial \bar{z}_1}$ is unbounded, the condition of Proposition~\ref{PoincareLelong} is still satisfied since
\[
\left|z_1 \frac{\partial g_{j\bar{k}}}{g_{j\bar{k}} \partial z_1}\right| = O\left(\frac{1}{\log|z_1|}\right) \to 0 \quad \text{as } z_1 \to 0.
\]

\noindent This gives the smooth part: $\theta_\varphi = dd^c\log\rho + dd^c\log|\psi|$.

\medskip
\noindent\textbf{Decomposition of the singular term.} The singular components of the Ricci curvature arise from the logarithmic singularity:
\[
\ddc \log\left(|z_1|^2 \log^2|z_1|^2\right) = \underbrace{\ddc \log|z_1|^2}_{2\pi[D]} +\underbrace{ \ddc \log|\log|z_1||^2}_{\theta'_\f}.
\]
Here the first term $\ddc\log|z_1|^2$ represents the current of integration $2\pi[D]$ along the divisor, while the second term $\theta'_\f = \ddc \log|\log|z_1||^2$ is the logarithmic correction. By Proposition~\ref{PoincareLelong}, $\theta'_\f$ is an $L^1_{\text{loc}}$-form. The two-sided estimate \(-C\om_{\rm mod}\leq\theta'_\f\leq C \om_{\rm mod}\) on $X\setminus D$ follows from direct calculation in local coordinates using the explicit formula for $\ddc\log|\log|z_1||^2$.

Since the singular part of $D$ has codimension at least two, the same decomposition extends globally by~\cite[Chapter III, Corollary 2.11]{Demailly}. This completes the proof.

\end{proof}

\section*{Acknowledgments}
The author would like to thank C.-M. Pan for helpful discussions and for bringing the work \cite{BJT} to the author's attention. The author is also grateful to A. Lahdili and S. Jubert for helpful discussions on weighted metrics, and to the supervisors Julien Keller and Hugues Auvray, without whom this work would not have been possible. This work was supported by the FRQNT grant ``M\'etriques k\"ahl\'eriennes sp\'eciales singuli\`eres et non-compactes'' (DOI: 10.69777/343263).

\newcommand{\etalchar}[1]{$^{#1}$}

\end{document}